\UseRawInputEncoding
\documentclass[12pt, reqno]{amsart}
\usepackage[margin=1in]{geometry}
\usepackage{amssymb,latexsym,amsmath,amscd,amsfonts}
\usepackage{latexsym}
\usepackage[mathscr]{eucal}
\usepackage{bm}
\usepackage{mathptmx}
\usepackage{amssymb}
\usepackage{amsthm}
\usepackage{dcolumn}
\usepackage[all]{xy}
\usepackage{enumitem}
\usepackage[utf8]{inputenc}

\def \qed {\hfill \vrule height6pt width 6pt depth 0pt}
\def\textmatrix#1&#2\\#3&#4\\{\bigl({#1 \atop #3}\ {#2 \atop #4}\bigr)}
\def\dispmatrix#1&#2\\#3&#4\\{\left({#1 \atop #3}\ {#2 \atop #4}\right)}
\newcommand{\beg}{\begin{equation}}
	\newcommand{\eeg}{\end{equation}}
\newcommand{\ben}{\begin{eqnarray*}}
	\newcommand{\een}{\end{eqnarray*}}

\newcommand{\C}{\mathbb C}

\newcommand{\N}{\mathbb N}

\newcommand{\E}{\mathbb E}

\newcommand{\Pe}{\mathbb P}
\newcommand{\D}{\mathbb D}
\newcommand{\G}{\mathbb G}

\newcommand{\lm}{\lambda}

\newcommand{\HS}{\mathcal{H}}

\newcommand{\al}{\alpha}
\newcommand{\DC}{\overline{\mathbb{D}}}

\newcommand{\la}{\langle}
\newcommand{\ra}{\rangle}
\newcommand{\LS}{\mathscr{L}}
\newcommand{\KS}{\mathcal{K}}

\newtheorem{thm}{Theorem}[section]
\newtheorem{cor}[thm]{Corollary}
\newtheorem{lem}[thm]{Lemma}

\newtheorem{prop}[thm]{Proposition}
\numberwithin{equation}{section} \theoremstyle{definition}
\newtheorem{defn}[thm]{Definition}

\def\textmatrix#1&#2\\#3&#4\\{\bigl({#1 \atop #3}\ {#2 \atop #4}\bigr)}
\def\dispmatrix#1&#2\\#3&#4\\{\left({#1 \atop #3}\ {#2 \atop #4}\right)}

\begin{document}
	
	\title[Function theoretic aspects of the symmetrized polydisc and generalization]{Function theoretic aspects of the symmetrized polydisc and generalization} 
	
	\author{SOURAV PAL AND NITIN TOMAR}
	
	\address[Sourav Pal]{Mathematics Department, Indian Institute of Technology Bombay,
		Powai, Mumbai - 400076, India.} \email{sourav@math.iitb.ac.in}
	
			\address[Nitin Tomar]{Statistics and Mathematics unit, Indian Statistical Institute Bangalore, Karnataka - 560059, India.} \email{tomarnitin414@gmail.com, nitin\_pd@isibang.ac.in}	
	
	
	
	\begin{abstract}
		We introduce Schur-Agler type class for the symmetrized polydisc $\mathbb{G}_d$ and establish a realization theorem for functions in this class. We prove an interpolation theorem on $\mathbb{G}_d$ with interpolating functions belonging to the associated Schur-Agler type class. Moreover, Toeplitz corona and extension theorems are established for $\mathbb{G}_d$. We also extend the realization, interpolation, Toeplitz corona and extension theorems to another domain $\Theta_d$, which is generalization of the domain $\mathbb{G}_d$.
	\end{abstract}	
	
	\maketitle
	
	\section{Introduction}
	
	\noindent The classical Nevanlinna-Pick interpolation theorem \cite{Nevanlinna, Pick} for the disc $\D=\{z\in \C: |z|<1\}$ states that given distinct points $z_1,\dots,z_n \in \D$ and target values $\lambda_1, \dotsc,\lambda_n \in \overline{\D}$, there exists a holomorphic function $f:\D\to\overline{\D}$ satisfying $f(z_i)=\lambda_i$ for $1\leq i \leq n$ if and only if the Pick matrix
	\[
	\bigl[(1-\lambda_i\overline{\lambda}_j)(1-z_i\overline{z}_j)^{-1}\bigr]_{i,j=1}^n
	\]
	is positive semi-definite. In the language of kernels, the positivity of Pick matrix is equivalent to saying that
	$ \bigl[(1 - \lambda_i \overline{\lambda}_j)k_S(z_i, z_j)\bigr]_{i,j=1}^n \geq 0$, where $k_S(z,w) = (1 - z\overline{w})^{-1}$ is the Szeg\H{o} kernel on $\D$. Recall that a kernel $k$ on a non-empty set $X$ is a positive semi-definite map $k: X \to X \to \C$ with $k(x, x) \ne 0$ for every $x \in X$. This kernel dependent formulation provides a natural framework for studying interpolation problems on more general domains in $\C^d$. For example, Abrahamse \cite{AbrahamseI} presented a similar interpolation theorem on an $m$-holed planar domain $R$ in terms of a family of Pick matrices. A more general approach to interpolation via families of kernels can be found in \cite{Jury}. A criterion for Nevanlinna-Pick interpolation problem on the bidisc $\D^2$ was presented in \cite{Agler_McCarthyI, Agler_McCarthy} via a certain family of  kernels on $\D^2$. These interpolation results naturally motivate the following interpolation problem. Suppose $z_1,\dotsc,z_n$ are distinct points in a domain $\Omega\subseteq\C^d$ and $\lambda_1,\dotsc,\lambda_n\in\overline{\D}$. Does there exist a holomorphic function $f:\Omega\to\C$ satisfying an appropriate norm condition such that $f(z_i)=\lambda_i$ for $1\leq i\leq n$? The solution to this problem depends crucially on the geometry of the underlying domain $\Omega$. A natural approach to the Nevanlinna-Pick interpolation problem on a domain is to first identify an appropriate class of interpolating functions together with an intrinsic characterization of that class. Such a characterization is commonly referred to as a \emph{realization theorem}. Two classes of functions that arise naturally in this context are the Schur class and the Schur-Agler class. For a domain $\Omega\subseteq\C^d$, the \emph{Schur class} $S(\Omega)$ is defined as
	\[
	S(\Omega)=\{f\in \text{Hol}(\Omega): \|f\|_{\infty,\Omega}\leq 1\},
	\]
	where $\text{Hol}(\Omega)$ is the class of all holomorphic functions on $\Omega$ and $\|f\|_{\infty, \Omega}=\sup\{|f(z)|: z \in \Omega\}$. A distinguished subclass of the Schur class of the polydisc $\D^d$ is the Schur-Agler class given by
	\[
	SA(\D^d)=\{g\in \text{Hol}(\D^d):\|g(T_1, \dotsc, T_d)\|\leq 1
		\ \text{for every commuting strict contractions} \ T_1, \dotsc, T_d\}.
	\]
	With these terminologies in place, the interpolating functions in the classical Nevanlinna-Pick theorem belong precisely to the Schur class $S(\D)$. A realization theorem for this class (see \cite{Agler_McCarthy}) asserts that $f \in S(\D)$ if and only if there exist a Hilbert space $\HS$ and a unitary operator
	\[
	V=
	\begin{bmatrix}
		A&B\\
		C&D
	\end{bmatrix}:\C\oplus\HS\to\C\oplus\HS
	\]
	such that $f(z)=A+zB(I_{\HS}-zD)^{-1}C$. Realization theorems for Schur class functions have subsequently been established for several other domains, including the annulus \cite{Drit2007_I} and the bidisc \cite{Agler1990, Agler_McCarthy}. For the polydisc $\D^d$, $d\ge3$, the situation is significantly different in the sense that a Schur class function need not admit a realization formula. However, a realization formula is guaranteed for the smaller class, namely Schur-Agler class $SA(\D^d)$, e.g., see \cite{Agler_McCarthy}. This viewpoint was further developed in the abstract setting in the works \cite{Drit2007_I, Drit2007_II, Ball_Guerra, Tirtha_Bis_Ch, Bis_Ch}, where the domain $\Omega$ is replaced by an arbitrary set $X$ and the Schur class is replaced by a class of functions determined by a prescribed family of test functions on $X$. Several domains arising as images of classical domains under proper holomorphic mappings have attracted considerable attention in multivariable operator theory and function theory. Among them, the symmetrized bidisc \cite{AglerYoung}, the tetrablock \cite{Abouhajar} and the pentablock \cite{AglerIV} play a prominent role. The symmetrized bidisc $\G_2$ is the image of the bidisc $\D^2$ under the symmetrization map
	\[
	(z_1,z_2)\mapsto (z_1+z_2,z_1z_2),
	\]
	and provides a fundamental example of a non-convex domain with rich interpolation theory. The tetrablock $\E$ and the pentablock $\Pe$ are further examples of domains whose function theory exhibits connections with operator theory and matrix-valued interpolation. The realization and interpolation theorems for the symmetrized bidisc $\G_2$ were obtained in \cite{AglerYoung2017, Tirtha_Sau} with different methods. For the tetrablock $\E$, the realization and interpolation theorems were obtained by the authors of \cite{Jain}. For the pentablock $\Pe$, the realization and interpolation theorems were presented in \cite{Pal2026}. An application of realization and interpolation theorems includes the extension problems for domains. Let $V$ be a nonempty subset of a domain $\Omega\subset\C^d$. A function $f:V\to\C$ is called holomorphic on $V$ if it admits a holomorphic extension to some neighbourhood of $V$. One may then ask whether every such function on $V$ extends holomorphically to the whole domain $\Omega$. In the bounded setting, a more subtle question is whether every bounded holomorphic function on $V$ has an extension belonging to $H^\infty(\Omega)$ that preserves its supremum norm. This question is referred to as the norm-preserving extension problem on $\Omega$. This problem has been explored for several important domains in $\C^d$, namely the bidisc \cite{Agler_McCarthy_2003}, the symmetrized bidisc \cite{Tirtha_Sau}, the tetrablock \cite{Jain} and the pentablock \cite{Pal2026}. 
	
	\smallskip 
	
	Realization theorems for function classes on domains in $\C^d$ also provide an important framework to study the Toeplitz corona problem. The classical corona theorem for the algebra $H^\infty(\D)$ was established by Carleson \cite{Carleson}. More precisely, if $\varphi_1,\dotsc,\varphi_d\in H^\infty(\D)$ satisfy
	\begin{equation}\label{eqn_corona}
		\sup_{z\in\D}\left[|\varphi_1(z)|^2+\cdots+|\varphi_d(z)|^2\right]\geq \epsilon^2
	\end{equation}
	for some $\epsilon>0$, then there exist functions $f_1,\dotsc,f_d\in H^\infty(\D)$ such that
	$
	\varphi_1f_1+\cdots+\varphi_df_d=1.
	$
	An operator-theoretic version of this result was obtained by Arveson \cite{Arveson_TC}, where the pointwise condition \eqref{eqn_corona} is replaced by the operator inequality
	$
	T_{\varphi_1}T_{\varphi_1}^*+\cdots+T_{\varphi_d}T_{\varphi_d}^*\geq \epsilon^2I.
	$
	Here, $T_{\varphi}$ denotes the Toeplitz operator with symbol $\varphi$. This result, known as the Toeplitz corona theorem, has subsequently been extended to several domains, including the bidisc \cite{Agler_McCarthyI}, the Euclidean unit ball and the polydisc \cite{Amar}, the symmetrized bidisc \cite{Tirtha_Sau_II}, the tetrablock \cite{Jain} and the pentablock \cite{Pal_penta_TC}. Thus, realization theorems play crucial role to address several central problems in function theory such as interpolation, extension and Toeplitz corona problems on domains.

	\smallskip 
	
	In this article, we present the vector-valued realization, interpolation and Toeplitz corona theorems for the symmetrized polydisc together with an extension theorem. The symmetrized polydisc $\G_d$ is a domain in $\C^d$ that is the image of the polydisc $\D^d$ under the symmetrization map $\pi_d=(s_1, \dotsc, s_d): \C^d \to \C^d$, where
		\[
	s_i(z_1, \dotsc, z_d)=\underset{1 \leq \ell_1<\dotsc <\ell_i  \leq d}{\sum}z_{\ell_1}\dotsc z_{\ell_i} \quad (1 \leq i \leq d).
	\]	
	The domain $\G_d$ originated from the spectral Nevanlinna-Pick interpolation problem, which is a matrix-valued analogue of the classical Nevanlinna-Pick interpolation problem. Given distinct points $\lm_1,\dotsc, \lm_n$ in $\D$ and $B_1, \dotsc, B_n \in M_d(\D)$, the problem asks for the existence of a holomorphic $d\times d$ matrix-valued function $F$ on $\D$ such that $F(\lm_i)=B_i$ for $1 \leq i \leq n$ and the spectral radius $r(F(\lm))$ of $F(\lm)$ is at most $1$ for every $\lm \in \D$. The geometry of the spectral unit ball in $M_d(\C)$ plays a central role in this problem and the symmetrized polydisc $\G_d$ naturally appears as a domain of significant interest in this direction. In fact, a matrix $B \in M_d(\C)$ is in the spectral unit ball, that is, its spectral radius $r(B)<1$ if and only if $\pi_d(\lambda_1, \dotsc, \lambda_d) \in \mathbb{G}_d$, where $\lambda_1, \dotsc, \lambda_d$ are the eigenvalues of $B$. For further details on the complex geometry of $\G_d$ and its relationship with the spectral Nevanlinna-Pick problem, we refer the reader to the works \cite{AglerI, Costara2005, Costara2005_II}.

	\smallskip  
	
	A central theme of this article is the development of a Schur-Agler class for $\G_d$ using operator-theoretic techniques. Our approach relies on Costara's characterization of $\G_d$ through a family of rational functions $\{\Phi_\alpha:\alpha\in\overline{\mathbb D}\}$. We first revisit this characterization and prove in Section \ref{sec_02} an equivalent pointwise criterion showing that a point $s=(s_1, \dotsc, s_{d-1}, p) \in \G_d$ if and only if $|\Phi_\alpha(s)|<1$ for every $\alpha \in \DC$. Although this characterization follows from the ideas in Costara's work \cite{Costara2005}, it does not appear explicitly in the literature and serves as a starting point for the operator-theoretic framework discussed in this paper. It naturally leads to an operator-valued analogue of these characterizations and motivates the introduction of a distinguished class $\mathfrak{M}_{\G_d}$ of commuting operator $d$-tuples associated with $\mathbb G_d$. We then introduce the notion of admissible kernels on $\G_d$, extending the classical Agler kernel framework to the present setting. These kernels provide the appropriate positivity conditions needed for reproducing kernel Hilbert space techniques necessary to establish realization theorem for the Schur-Agler class for $\G_d$. En route, several auxiliary results are developed in Section \ref{sec_02} including the boundedness of the coordinate multiplication linear maps on reproducing kernel Hilbert spaces associated with admissible kernels and transfer-function realizations arising from unitary colligations. As an application of the realization theorem, we first establish in Section \ref{sec_03} a Nevanlinna-Pick type interpolation theorem for $\G_d$. In fact, we characterize the existence of an interpolating function in the Schur-Agler class for $\G_d$ through positivity conditions involving admissible kernels, thereby extending interpolation theory on $\G_2$ from \cite{Tirtha_Sau} to the higher-dimensional domain $\G_d$. The realization and interpolation theorems provide the main tools for the remainder of the paper. Using this, we prove Toeplitz corona and extension theorems for $\G_d$ in Section \ref{sec_03}. We also consider the family of domains
	\[
	\Theta_d=
	\left\{
	\bigl(s_1(z_1^m,\ldots,z_d^m),\ldots,s_{d-1}(z_1^m,\ldots,z_d^m),(z_1\cdots z_d)^{m/p}\bigr):
	z_1,\ldots,z_d\in\mathbb D
	\right\},
	\]
	introduced by the authors of \cite{Biswas}, where $p$ is a fixed divisor of given $m \in \N$. This family contains the symmetrized polydisc as the special case $m=p=1$. We present in Section \ref{sec_04} realization, interpolation, Toeplitz corona and extension theorems for $\Theta_d$ by adapting the corresponding results established for $\mathbb G_d$. Since the proofs are essentially identical, we only provide the necessary modifications and state the corresponding results.
	
	\smallskip 
	
\noindent \textbf{Notations.}
Throughout the paper, all Hilbert spaces are assumed to be complex and separable, and all operators are bounded and linear. We use $\HS$, $\KS$, and $\LS$ to denote Hilbert spaces. For Hilbert spaces $\HS$ and $\LS$, let $\mathcal{B}(\HS,\LS)$ denote the space of bounded linear operators from $\HS$ to $\LS$, and write $\mathcal{B}(\HS)=\mathcal{B}(\HS,\HS)$. For a compact set $K$, let $C(K)$ denote the $C^*$-algebra of complex-valued continuous functions on $K$, equipped with the supremum norm $\|\cdot\|_{\infty,K}$. For domains $\Omega$ and $\Omega'$, we write $\operatorname{Hol}(\Omega,\Omega')$ for the space of holomorphic functions from $\Omega$ into $\Omega'$. For functions $g, k: \Omega \times \Omega \to \mathcal{B}(\HS)$, we denote $g\oslash k$ by the $\mathcal{B}(\HS \otimes \HS)$-valued functions on $\Omega \times \Omega$ defined as $g\oslash k(z, w)=g(z, w)\otimes k(z, w)$. Finally, for a commuting tuple $\underline{T}=(T_1,\ldots,T_d)$ of operators, $\sigma_T(\underline{T})$ denotes its Taylor joint spectrum.
	
	\section{Schur-Agler class and Realization theorem for the symmetrized polydisc}\label{sec_02}
	
	\noindent In this section, we introduce a Schur-Agler type class for the symmetrized polydisc $\G_d$ and present a realization theorem for functions belonging to this class. We begin by recalling the necessary background from \cite{Costara2005} concerning the characterization of points in $\G_d$. For $s=(s_1, \dotsc, s_{d-1}, p) \in \C^d$, consider the polynomials $Q(s)$ and $R(s)$ given by
	\begin{align}\label{eqn_QR}
		Q(s)(\alpha)&=d(-1)^dp\alpha^{d-1}+(d-1)(-1)^{d-1}s_{d-1}\alpha^{d-2}+\dotsc+(-s_1),  \notag \\
		R(s)(\alpha)&=d-(d-1)s_1\alpha+\dotsc+(-1)^{d-1}s_{d-1}\alpha^{d-1}.
	\end{align} 
	Using $Q(s)$ and $R(s)$, we define $\mathrm{J}(s)(\alpha)=Q(s)(\alpha)\slash R(s)(\alpha)$. We have by Theorem 3.1 in \cite{Costara2005} that $(s_1, \dotsc, s_{d-1}, p) \in \G_d$ if and only if $\|\mathrm{J}(s)\|_{\infty, \DC} <1$. For every $\alpha \in \DC$ and $s=(s_1, \dotsc, s_{d-1}, p) \in \C^d$, let us define
	\begin{align}\label{eqn_J(z)}
		\Phi_{\alpha}(s_1, \dotsc, s_{d-1}, p)=\mathrm{J}(s)(\alpha)=\frac{d(-1)^dp\alpha^{d-1}+(d-1)(-1)^{d-1}s_{d-1}\alpha^{d-2}+\dotsc+(-s_1)}{d-(d-1)s_1\alpha+\dotsc+(-1)^{d-1}s_{d-1}\alpha^{d-1}}.
	\end{align}
	Consequently, 
	\[
	(s_1, \dotsc, s_{d-1}, p) \in \G_d \quad  \text{if and only if} \quad \underset{\alpha \in \DC}{\sup}|\Phi_\alpha(s_1, \dotsc, s_{d-1}, p)|<1.
	\] 
	We now state the following characterization of points in $\G_d$. To the best of our knowledge, it has not appeared explicitly in the literature and we present it here for the sake of completeness. The proof is based on the ideas
	of Costara \cite{Costara2005}. 
	
	\begin{thm}\label{prop_char}
		Let $(s_1, \dotsc, s_{d-1}, p) \in \G_d $and let $\{\Phi_\alpha:\alpha\in\DC\}$ denote the family of functions defined in \eqref{eqn_J(z)}. Then $(s_1, \dotsc, s_{d-1}, p) \in \G_d$ if and only if $|\Phi_\alpha(s_1, \dotsc, s_{d-1}, p)|<1$ for every $\alpha \in \DC$. 
	\end{thm}
	
	\begin{proof}
		By above discussion, the necessary part follows trivially. Let $s=(s_1, \dotsc, s_{d-1}, p)$ and $\mathrm{J}(s)(\alpha)=\Phi_\alpha(s_1, \dotsc, s_{d-1}, p)$ for $\alpha \in \DC$. Assume that $|\Phi_\alpha(s_1, \dotsc, s_{d-1}, p)|<1$ for every $\alpha \in \DC$. In view of the preceding discussion, it suffices to show that
		\[
		\underset{\alpha \in \DC}{\sup}|\Phi_\alpha(s_1, \dotsc, s_{d-1}, p)|<1, \ \text{that is,} \ \|J(s)\|_{\infty, \DC}<1
		\]
		for then it follows that $s \in \G_d$. Following the proof of Theorem 3.1 in \cite{Costara2005}, we have that if there exists some $\alpha_0 \in \DC$ satisfying $R(s)(\alpha_0)=0$, then $|\mathrm{J}(s)(\alpha_0)| \geq 1$. We briefly discuss it here for the sake of completeness.  Define
		$
		P(s)(\alpha)=\alpha^d-s_1\alpha^{d-1}+\dotsc+(-1)^dp.
		$
		It is not difficult to see that 
		\[
		Q(s)(\alpha)=\frac{d}{d\alpha}(\alpha^dP(s)(1\slash \alpha)) \quad \text{and} \quad R(s)(\alpha)=\alpha^{d-1}P'(s)(1\slash \alpha),
		\]
		where $P'(s)(1\slash \alpha)$ is the derivative of  $P(s)(\alpha)$ evaluated at the point $1\slash \alpha$. Suppose $R(s)(\alpha_0)=0$ for some $\alpha_0 \in \DC$. Then $\alpha_0 \ne 0$ and $Q(s)(\alpha_0)=0$ as $|\Phi_{\alpha_0}(s_1, \dotsc, s_{d-1}, p)|<1$. So, $P'(s)(1\slash \alpha_0)=0$ and $d\alpha_0^{d-1}P(s)(1\slash \alpha_0)-\al_0^{d-2}P'(s)(1\slash \alpha_0)=0$. Thus, $P(s)(1\slash \alpha_0)=0$ and $P(s)(\alpha)$ has $1\slash \alpha_0$ as a zero of order at least two. One can write $P(s)(\alpha)=(\alpha-1\slash \alpha_0)^m h(s)(\alpha)$ on $\C$, where $m \geq 2$ and $h(s)$ is non-zero on a neighbourhood of $1\slash \alpha_0$. Then for $\alpha$ in a sufficiently small neighbourhood of $\alpha_0$, we have
		\begin{align*} 
			\mathrm{J}(s)(\alpha)
			=\frac{Q(s)(\alpha)}{R(s)(\alpha)}
			&=\frac{d\alpha^{d-1}P(s)(1\slash \alpha)-\al^{d-2}P'(s)(1\slash \alpha)}{\alpha^{d-1}P'(s)(1\slash \alpha)}\\
			&=\frac{d\al^{d-1}(1\slash \al-1\slash \al_0)^mh(s)(1\slash \al)}{m\al^{d-1}(1\slash \al-1\slash \al_0)^{m-1}h(s)(1\slash \al)+\al^{d-1}(1\slash \al-1\slash \al_0)^mh'(s)(1\slash \al)}\\
			&\quad -\frac{\al^{d-2}\left[m(1\slash \al-1\slash \al_0)^{m-1}h(s)(1\slash \al)+(1\slash \al-1\slash \al_0)^mh'(s)(1\slash \al)\right]}{m\al^{d-1}(1\slash \al-1\slash \al_0)^{m-1}h(s)(1\slash \al)+\al^{d-1}(1\slash \al-1\slash \al_0)^mh'(s)(1\slash \al)}
		\end{align*} 
		and so,
		\[ 
		|\mathrm{J}(s)(\alpha_0)|=\left|\frac{-m\alpha_0^{d-2}h(s)(1\slash \alpha_0)}{m\alpha_0^{d-1}h(s)(1\slash \alpha_0)}\right|=|1\slash \alpha_0| \geq 1,
		\]
		which contradicts the hypothesis on $\mathrm{J}(s)(\alpha_0)$. Thus, $R(s)$ is non-zero on $\DC$. Therefore, the rational function $\mathrm{J}(s)=Q(s)\slash R(s)$ is holomorphic on $\D$ and continuous on $\DC$ that satisfies $|J(s)(\alpha)|<1$ for every $\alpha \in \DC$. Thus, $\|\mathrm{J}(s)\|_{\infty, \DC}=|J(s)(\alpha)|$ for some $\alpha \in \DC$ and so, $\|\mathrm{J}(s)\|_{\infty, \DC}<1$.
	\end{proof}
	
	Using the description of $\G_d$ as in Theorem \ref{prop_char} together with the discussion preceding it, we have that the map $\mathrm{J}(s): \DC \to \C$ is a well-defined continuous function on $\DC$  that satisfies $\|\mathrm{J}(s)\|_{\infty, \DC}<1$ for every $s=(s_1, \dotsc, s_{d-1}, p) \in \G_d$. Next, we introduce an operator-theoretic analog of the inequalities provided in Theorem \ref{prop_char}. Let $(S_1, \dotsc, S_{d-1}, P)$ be a commuting $d$-tuple of operators acting on a Hilbert space $\HS$ such that its joint spectrum $\sigma_T(S_1, \dotsc, S_{d-1}, P) \subseteq \G_d$. Fix $\alpha \in \DC$ and set
	\[
	R(S_1, \dotsc, S_{d-1}, P)(\alpha)=dI-(d-1)\alpha S_1+\dotsc+(-1)^{d-1}\alpha^{d-1}S_{d-1}.
	\]
	Let $\lambda \in \sigma(R(S_1, \dotsc, S_{d-1}, P)(\alpha))$. By spectral mapping principle, $\lambda=R(s_1, \dotsc, s_{d-1}, p)(\alpha)$ for some $(s_1, \dotsc, s_{d-1}, p) \in \G_d$. Following the proof of Theorem \ref{prop_char}, we have that $R(s_1, \dotsc, s_{d-1}, p)$ is non-vanishing on $\DC$ and thus, $\lambda \ne 0$. Therefore, $R(S_1, \dotsc, S_{d-1}, P)(\alpha)$ is an invertible operator. By rational functional calculus, we have that $\Phi_{\alpha}(S_1, \dotsc, S_{d-1}, P)$ equals
	{\small
		\begin{align*}
			\left(d(-1)^d\alpha^{d-1}P+(d-1)(-1)^{d-1}\alpha^{d-2}S_{d-1}+\dotsc+(-S_1)\right)\left(dI-(d-1)\alpha S_1+\dotsc+(-1)^{d-1}\alpha^{d-1}S_{d-1}\right)^{-1}.
		\end{align*}
	}
	\par \noindent We consider the class of commuting $d$-tuples of operators given by 
	\[
	\mathfrak{M}_{\G_d}=\left\{\underline{S}=(S_1, \dotsc, S_{d-1}, P): \sigma_T(\underline{S}) \subseteq \G_d, \ \|\Phi_\al(\underline{S})\|<1 \ \text{for every} \ \al \in \DC\right\}.
	\]
	The inequalities in $\mathfrak{M}_{\G_d}$ are an operator theoretic analog of inequalities from Theorem \ref{prop_char} in which the scalars are replaced by commuting operators subjected to similar norm bounds. We put forth some useful properties of $\mathfrak{M}_{\G_d}$, which will be used throughout the article.
	\begin{enumerate}[leftmargin=*]
		
		\item $\mathfrak{M}_{\G_d}$ contains $\G_d$ in the sense that $(s_1I, \dotsc, s_{d-1}I, pI) \in \mathfrak{M}_{\G_d}$ for every $(s_1, \dotsc, s_{d-1}, p) \in \G_d$.		
		\item $\mathfrak{M}_{\G_d}$ is closed under adjoint, i.e., if $(S_1, \dotsc, S_{d-1}, P) \in \mathfrak{M}_{\G_d}$ then $(S_1^*, \dotsc, S_{d-1}^*, P^*) \in \mathfrak{M}_{\G_d}$.
		
		\item Evidently, the domain $\G_d$ is $(1,\dotsc, d-1, d)$-quasi-balanced, i.e.,
		$(r s_1, \dotsc, r^{d-1}s_{d-1}, r^dp)\in\G_d$ for every $r \in (0, 1)$ and $(s_1,\dotsc,s_{d-1},p)\in\G_d$. Therefore, 
		$(rS_1, \dotsc, r^{d-1}S_{d-1}, r^dP) \in \mathfrak{M}_{\G_d}$ for $0 \leq r \leq 1$ and $(S_1, \dotsc, S_{d-1}, P) \in \mathfrak{M}_{\G_d}$.
	\end{enumerate}
	
	Next, we introduce a vector-valued Schur-Agler class for the symmetrized polydisc. To do so, we recall from \cite{Ambrozie, Esch_Put} the vector-valued holomorphic functional calculus. For Hilbert spaces $\LS, \LS'$, let $f: \G_d \to \mathcal{B}(\LS, \LS')$ be a holomorphic function. Suppose $\underline{S}=(S_1, \dotsc, S_{d-1}, P)$ is a commuting tuple of operators on a Hilbert space $\HS$ with $\sigma_T(\underline{S}) \subseteq \G_d$. Then $f(\underline{S}): \HS \otimes \LS \to \HS \otimes \LS'$ is a bounded linear map. Given an open neighbourhood $U_0$ of $\sigma_T(\underline{S})$, there is a unique continuous linear map $g \mapsto g(\underline{S})$ from $\text{Hol}(U_0, \mathcal{B}(\LS, \LS')) \equiv \text{Hol}(U_0)\otimes \mathcal{B}(\LS, \LS')$ into $\mathcal{B}(\HS \otimes \LS, \HS \otimes \LS')$ taking $g \otimes A$ into $g(\underline{S}) \otimes A$ for $g \in \text{Hol}(U_0)$ and $A \in \mathcal{B}(\LS, \LS')$. In particular, it follows that 
	\[
	\la f(\underline{S})(h_1 \otimes x), h_2 \otimes y\ra=\la f_{x, y}(\underline{S})h_1, h_2 \ra 
	\]
	for every $h_1, h_2 \in \HS, x \in \LS$ and $y \in \LS'$, where $f_{x, y}: \G_d \to \C$ is the scalar-valued holomorphic map defined by $f_{x, y}(s)=\la f(s)x, y\ra$. Furthermore, it is clear from the definition of $\mathfrak{M}_{\G_d}$ that the functional calculus $f(\underline{S})$ is well defined for all $f \in \mathrm{Hol}(\G_d)$ and $\underline{S} \in \mathfrak{M}_{\G_d}$.  
	
	\begin{defn}
		For Hilbert spaces $\LS, \LS'$, the \textit{$\mathcal{B}(\LS, \LS')$-valued Schur-Agler class} is defined as
		\[
		SA_{\G_d}(\LS, \LS')=\left\{f: \G_d \to \mathcal{B}(\LS, \LS'): \ \text{$f$ is holomorphic and } \ \|f(\underline{S})\| \leq 1 \ \text{for every $\underline{S} \in \mathfrak{M}_{\G_d}$}\right\}.
		\] 
	\end{defn}	
	When $\LS=\LS'=\C$, then we denote the Schur-Agler class $SA_{\G_d}(\C, \C)$ simply by $SA(\G_d)$, which is the collection of all complex-valued holomorphic functions on $\G_d$ that satisfies $\|f(\underline{S})\| \leq 1$ for every $\underline{S} \in \mathfrak{M}_{\G_d}$. We mention here that $SA(\G_2)$ is the Schur-class $S(\G_2)$ as discussed in \cite{AglerYoung2017, AglerYoung, Tirtha_Sau}.	We now define the notion of admissible kernels for the symmetrized polydisc.
	
	\begin{defn} Let $\LS$ be a Hilbert space and let $F$ be a subset of $\C^d$. A $\mathcal{B}(\LS)$-valued weak kernel on $F$ is a positive semi-definite function $k: F \times F \to \mathcal{B}(\LS)$, that is, 
		\[
		\overset{n}{\underset{i, j=1}{\sum}}\la k(s^{(i)}, s^{(j)})v_j, v_i\ra_{\LS} \geq 0 
		\]
		for every $\{s^{(1)}, \dotsc, s^{(n)}\} \subset F$ and $v_1, \dotsc, v_n \in \LS$. In addition, if $k(z, z) \ne 0$ for all $z \in F$, we say that $k$ is a $\mathcal{B}(\LS)$-valued kernel. A $\mathcal{B}(\LS)$-valued kernel (or, weak kernel) $k$ on a subset of $\G_d$ is said to be \textit{admissible} if
		\[
		(s, t) \mapsto \left(1-\Phi_{\al}(s)\overline{\Phi_{\al}(t)}\right)k(s, t)
		\]
		is a positive semi-definite map for all $\alpha \in \DC$. The class of scalar-valued admissible kernels on $\G_d$ is denoted by $AK(\G_d)$. For a subset $F$ of $\G_d$, a map $\xi: F \times F \to \mathcal{B}(C(\DC), \mathcal{B}(\LS))$ is called \textit{completely positive} if for every $n \in \N, \{v_1, \dotsc, v_n\} \subset \LS, \{s^{(1)}, \dotsc, s^{(n)}\} \subset F$ and $\{h_1, \dotsc, h_n\} \subset C(\DC)$, we have
		\[
		\overset{n}{\underset{i, j=1}{\sum}}\la \xi(s^{(i)}, s^{(j)})(h_i\overline{h}_j)v_j, v_i\ra_\LS \geq 0.
		\]
	\end{defn}
	If $\LS=\C$, then $\mathcal{B}(C(\DC), \LS)$ is the dual space $C(\DC)^*$ of $C(\DC)$.  We denote by $C(\DC)_F^+$ the set of all completely positive maps with values in $C(\DC)^*$, and the set of all positive semi-definite functions $\delta: F \times F \to \C$ is denoted by $\C_F^+$. 	
	The following result provides the first step in the proof of the realization theorem for $\G_d$. Its proof follows the standard arguments from \cite{Agler_McCarthy, Tirtha_Sau, Jain, Pal_penta}.	
	
	\begin{prop}\label{prop_204}
		Let $F$ be a subset of $\G_d$ and let $\xi: F \times F \to  \mathcal{B}(C(\DC), \mathcal{B}(\LS))$ be a completely positive map. Then there exist a Hilbert space $\HS$, a function $L: F \to \mathcal{B}(C(\DC), \mathcal{B}(\HS, \LS))$ and a unital $*$-representation $\rho: C(\DC) \to \mathcal{B}(\HS)$ such that
		\[
		\xi(s, t)(f\overline{g})=L(s)(f)(L(t)(g))^* \quad \text{and} \quad L(s)(fg)^*=\rho(f)^*L(s)(g)^*
		\]	
		for all $f, g \in C(\DC)$ and $s, t \in F$. 
	\end{prop} 
	
	\begin{proof}
		We begin by proving the result in the case $\LS=\C$. Consider the map
		\[
		\xi': (F \times C(\DC)) \times (F \times C(\DC)) \to \C \quad \text{given by} \quad \xi'((s, h_1), (t, h_2))=\xi(s, t)(h_1\overline{h}_2).
		\] 
		It is not difficult to see that $\xi'$ is a positive semi-definite function on $F \times C(\DC)$. By Theorem 2.53 in \cite{Agler_McCarthy}, there is a Hilbert space $\mathcal{K}$ and a vector-valued map $\eta: F \times C(\DC) \to \mathcal{K}$ such that 
		\[
		\langle \eta(s, h_1), \eta(t, h_2)\rangle_{\mathcal{K}}=\xi'((s, h_1), (t, h_2))=\xi(s, t)(h_1\overline{h}_2)
		\]
		for all $h_1, h_2 \in C(\DC)$ and  $s, t \in F$. Indeed, one can choose the Hilbert space $\KS=\overline{\text{span}}\{\eta(s, h): s \in F, h \in C(\DC)\}$. Let us define
		$\ell: F \to \mathcal{B}(C(\DC), \KS)$ as $\ell(s)(h)=\eta(s, h)$. Then 
		\[
		\xi(s, t)(f\overline{g})=\langle \ell(s)f, \ell(t)g \rangle  \ \text{and so,}  \ \|\ell(s)(f)\|^2=\|\eta(s, f)\|^2=\|\xi(s, s)(f\bar{f})\| \leq \|\xi(s, s)\| \|f\|^2_{\infty, \DC}
		\]
		for all $s, t \in F$ and $f, g \in C(\DC)$. Consider the map $\rho_0: C(\DC) \to \mathcal{B}(\KS)$ given by $\rho_0(h_1)\eta(s, h_2)=\eta(s, h_1h_2)$. It is easy to see that $\rho_0$ is a unital $*$-representation such that $\rho_0(f)\ell(s)(g)=\ell(s)(fg)$. 
		
		\medskip 
		
		For the general case, let $\{e_\lambda: \lm \in \Lambda \}$ be an orthonormal basis for the Hilbert space $\LS$. Consider
		\[
		\Xi: (F \times \Lambda) \times (F \times \Lambda) \to C(\DC)^* \quad \text{defined as} \quad \Xi((s, \lm_1), (t, \lm_2))(h)=\la e_{\lm_1}, \xi(s, t)(\overline{h})e_{\lm_2} \ra. 
		\]
		It is not difficult to see that $\Xi$ is a $C(\DC)^*$-valued positive semi-definite map on $F \times \Lambda$. We can now apply the first case to guarantee the existence of a Hilbert space $\HS$, a function $\ell_\Lambda: F \times \Lambda \to \mathcal{B}(C(\DC), \mathcal{H})$ and a unital $*$-representation $\rho: C(\DC) \to \mathcal{B}(\HS)$ such that 
		\[
		\Xi((s, \lm_1), (t, \lm_2))(h_1\overline{h}_2)=\la \ell_\Lambda(s, \lm_1)h_1, \ell_\Lambda(t, \lm_2)h_2 \ra_{\HS} \quad \text{and} \quad \rho(h_1)\ell_\Lambda(s, \lm_1)(h_2)=\ell_\Lambda(s, \lm_1)(h_1h_2).
		\]
		Finally, consider the function $L: F \to \mathcal{B}(C(\DC), \mathcal{B}(\HS, \LS))$ defined as $L(s)(h)^*(e_\lm)=\ell_\Lambda(s, \lm)(\overline{h})$. For
		$x=\underset{\lambda \in \Lambda}{\sum} c_\lambda e_\lambda \in \LS$, we have
		\begin{align*}
			\|L(s)(h)^*x\|_{\HS}^2
			=\left\|\sum_{\lambda\in\Lambda}c_\lambda
			L(s)(h)^*e_\lambda\right\|_{\HS}^2
			&=\sum_{\lambda,\mu}
			c_\lambda\overline{c}_\mu
			\left\langle
			\ell_\Lambda(s,\lambda)(\overline{h}),
			\ell_\Lambda(s,\mu)(\overline{h})
			\right\rangle_{\HS}\\
			&=\sum_{\lambda,\mu}
			c_\lambda\overline{c}_\mu
			\left\langle
			e_\lambda,
			\xi(s,s)(|h|^2)e_\mu
			\right\rangle_{\LS}\\
			&=\left\langle
			x, \xi(s,s)(|h|^2)x
			\right\rangle_{\LS}\\
			&\leq
			\|\xi(s,s)(|h|^2)\|\,
			\|x\|_{\LS}^2
		\end{align*}
		and so, $L(s):C(\overline{\mathbb D})\to\mathcal B(\HS,\LS)$ is a bounded linear map for every $s \in F$. Note that 
		\[
		\rho(h_1)^*L(s)(h_2)^*e_\lm=\rho(\overline{h}_1)\ell_\Lambda(s, \lm)(\overline{h}_2)=\ell_\Lambda(s, \lm)(\overline{h}_1\overline{h}_2)=L(s)(h_1h_2)^*e_\lm
		\] 
		for every $\lm \in \Lambda$ and so, $\rho(h_1)^*L(s)(h_2)^*=L(s)(h_1h_2)^*$. Let $s, t \in F$ and $h_1, h_2 \in C(\DC)$. Then 
		\begin{align*}
			\la e_{\lm_1}, \xi(s, t)(\overline{h}_1h_2)e_{\lm_2},\ra_{\LS}
			=\Xi((s, \lm_1)(t, \lm_2))(h_1\overline{h}_2)
			&= \la \ell_\Lambda(s, \lm_1)h_1, \ell_\Lambda(t, \lm_2)h_2 \ra_{\HS}\\
			&=\la L(s)(\overline{h}_1)^*e_{\lm_1}, L(t)(\overline{h}_2)^*e_{\lm_2} \ra_{\HS}\\
			&=\la e_{\lm_1}, L(s)(\overline{h}_1)L(t)(\overline{h}_2)^*e_{\lm_2} \ra_{\HS}
		\end{align*}
		for every $\lm_1, \lm_2 \in \Lambda$ and thus, $\xi(s, t)(\overline{h}_1h_2)=L(s)(\overline{h}_1)L(t)(\overline{h}_2)^*$. The proof is now complete.
	\end{proof}	
	
	Next, we describe the structure of the self-adjoint functions on subsets of $\G_d$ that preserve positivity under tensoring with all vector-valued admissible weak kernels, which plays a crucial role in the proof of realization theorem. The arguments are based on the similar ideas as those used in the proof of Lemma 5.1 in \cite{Tirtha_Sau_II} (also see \cite{Tirtha_Bis_Ch}). We briefly discuss it here the parts that need necessary modifications. Let $F \subseteq \G_d$ and let $g: F \times F \mapsto \mathcal B(\LS)$ be a continuous self-adjoint function, that is, $g(s,t)=g(t,s)^*$ for every  $s, t \in F$. Suppose the map $g \oslash k: F \times F \to \mathcal{B}(\LS \otimes \LS)$ given by 
	\[
	g\oslash k(s,t)= g(s,t)\otimes k(s,t)
	\]
	is positive semi-definite for every $\mathcal B(\LS)$-valued admissible weak kernel
	$k$ on $F$. We first consider the finite subsets of $F$. Let $F_n=\{s^{(1)}, \dotsc, s^{(n)}\}\subseteq F$ and let $\mathfrak W$ denote the collection of all self-adjoint $n\times n$
	operator matrices of the form
	\[
	\left[
	\xi(s^{(i)},s^{(j)})
	\left(
	1-\mathrm{J}(s^{(i)})\overline{\mathrm{J}(s^{(j)})}
	\right)
	\right]_{i,j=1}^n,
	\]
	where
	$
	\xi:F_n\times F_n
	\longrightarrow
	\mathcal B(C(\overline{\mathbb D}),\mathcal B(\LS))
	$
	is completely positive.  Recall that a subset $W$ of a real vector space is called a
	\emph{wedge} if it is closed under addition and multiplication by non-negative
	scalars. It is immediate that $\mathfrak{W}$ is a wedge in the real vector space of
	self-adjoint $n\times n$ operator matrices with entries in $\mathcal{B}(\LS)$. Since $\mathcal B(\LS^n)$ is the dual of the trace class operators
	$\mathcal B_1(\LS^n)$, it is equipped with the canonical weak-$*$ topology.
	We also equip the wedge $\mathfrak W$ with the relative weak-$*$ topology inherited
	from $\mathcal B(\LS^n)$. We first show that the wedge $\mathfrak W$ is weak-$*$ closed. Let
	\[
	w_\beta=[w_\beta(s^{(i)}, s^{(j)})]_{i, j=1}^n=
	\left[
	\xi_\beta(s^{(i)},s^{(j)})
	\left(
	1-\mathrm{J}(s^{(i)})\overline{\mathrm{J}(s^{(j)})}
	\right)
	\right]_{i,j=1}^n
	\]
	be a net in $\mathfrak W$ converging in the weak-$*$ topology to
	$w=[w_{ij}]_{i, j=1}^n \in\mathcal B(\LS^n)$. Recall that $\mathcal B(\LS^n)=\mathcal B_1(\LS^n)^*$, the weak-$*$ convergence means that
	$
	\operatorname{tr}(w_\beta X)\longrightarrow
	\operatorname{tr}(wX)
	$
	for every $X\in\mathcal B_1(\LS^n)$. Since each $w_\beta$ is self-adjoint, the limit
	$w$ is also self-adjoint. Fix $i, j$ with $1\le i,j\le n$ and $u,v\in\LS$. Let $X$ be the trace class
	operator matrix whose $(j,i)$-th entry is $u\otimes v$ and remaining
	entries are zero. Then $\la w_\beta(s^{(i)}, s^{(j)})u, v\ra$ converges to $\la w_{ij}u, v\ra$. Since $w_\beta \to w$ in the weak-$*$ topology of 
	$\mathcal{B}(\LS^n)$, the net $\{w_\beta\}$ is bounded in operator norm. 
	Thus, for $1\leq i\leq n$, there exists $C_i>0$ such that
	$
	\|w_\beta(s^{(i)},s^{(i)})\|\leq C_i
	$
	for every $\beta$. Since $\sup\{|\Phi_\alpha(s)| : \alpha \in \DC\}<1$ for every  $s\in\G_d$ and $F_n$ is finite, there exists $\varepsilon>0$ such that $1-|\mathrm J(s^{(i)})|^2\ge\varepsilon 1$ for $1 \leq i \leq n$. By completely positivity of $\xi_\beta$, we have that $\la w_\beta(s^{(i)}, s^{(i)})u, u \ra \geq \varepsilon
	\left\langle
	\xi_\beta(s^{(i)},s^{(i)})(1)u,u
	\right\rangle$ for every $u \in \LS$ and so, 
	$0\leq \varepsilon\xi_\beta(s^{(i)},s^{(i)})(1)
	\leq w_\beta(s^{(i)},s^{(i)})$. Consequently, 
	$
	\|\xi_\beta(s^{(i)},s^{(i)})(1)\|
	\leq C_i \slash \varepsilon.
	$
	Let $h\in C(\overline{\mathbb D})$. Since $|h|^2\leq \|h\|_\infty^2 1$,
	we have by the positivity of $\xi_\beta(s^{(i)}, s^{(i)})$ that
	$
	\xi_\beta(s^{(i)},s^{(i)})(|h|^2)
	\leq
	\|h\|_\infty^2
	\xi_\beta(s^{(i)},s^{(i)})(1)
	$
	and so, $
	\left\|
	\xi_\beta(s^{(i)},s^{(i)})(|h|^2)
	\right\|
	\leq
	(C_i\slash \varepsilon)\|h\|_\infty^2$. It follows from the Cauchy-Schwarz inequality for completely positive maps that for every $h \in C(\DC), u, v \in \LS$ and $1 \leq i, j \leq n$, 
	\[
	\begin{aligned}
		\left|
		\left\langle
		\xi_\beta(s^{(i)},s^{(j)})(h)u,v
		\right\rangle
		\right|^2
		\leq
		\left\langle
		\xi_\beta(s^{(i)},s^{(i)})(|h|^2)u,u
		\right\rangle
		\left\langle
		\xi_\beta(s^{(j)},s^{(j)})(1)v,v
		\right\rangle .
	\end{aligned}
	\]
	Therefore,
	\[
	\left|
	\left\langle
	\xi_\beta(s^{(i)},s^{(j)})(h)u,v
	\right\rangle
	\right|
	\leq
	\frac{\sqrt{C_iC_j}}{\varepsilon}
	\|h\|_\infty
	\|u\|\|v\| \quad \text{and so,} \quad 
	\|\xi_\beta(s^{(i)},s^{(j)})(h)\|
	\leq
	\frac{\sqrt{C_iC_j}}{\varepsilon}
	\|h\|_\infty .
	\]
	Thus,
	$
	\|\xi_\beta(s^{(i)},s^{(j)})\|
	\leq
	\sqrt{C_iC_j} \slash \varepsilon
	$
	for every $\beta$ and $1 \leq i,j \leq n$. Therefore the nets
	$
	\{\xi_\beta(s^{(i)},s^{(j)})\}_\beta
	$
	are uniformly bounded in
	$
	\mathcal B(C(\overline{\mathbb D}),\mathcal B(\LS)).
	$
	Since 
	$
	\mathcal B(C(\overline{\mathbb D}),\mathcal B(\LS))
	$
	is the dual of a Banach space, its closed norm balls are weak-$*$ compact by the
	Banach-Alaoglu theorem and so, we have a weak-$*$ convergent subnet of $\beta$ for every pair $(i, j)$ with $1\leq i, j \leq n$. Since there are only finitely many such pairs $(i,j)$, one can choose a common subnet and we continue to denote it by $\beta$. Then $\xi_\beta(s^{(i)},s^{(j)})
	\xrightarrow{w^*} \xi(s^{(i)},s^{(j)})$ in $\mathcal B(C(\overline{\mathbb D}),\mathcal B(\LS))$ for every
	$1\leq i,j\leq n$. Define
	\[
	\xi:F_n\times F_n\longrightarrow
	\mathcal B(C(\overline{\mathbb D}),\mathcal B(\LS))
	\]
	by these weak-$*$ limits. Since complete positivity is preserved under pointwise
	weak-$*$ limits, it follows that $\xi$ is a completely positive map. Note that
	\[
	\mathcal B(C(\overline{\mathbb D}),\mathcal B(\LS))
	\cong
	\bigl(
	C(\overline{\mathbb D})
	\widehat{\otimes}_\pi
	\mathcal B_1(\LS)
	\bigr)^*,
	\]
	where $\widehat{\otimes}_\pi$ is the projective tensor product of Banach spaces.
	Fix $1\leq i,j\leq n$ and set
	$
	h_{ij}
	=
	1-\mathrm J(s^{(i)})\overline{\mathrm J(s^{(j)})}
	$.
	Since the evaluation map
	\[
	E_{h_{ij}}:
	\mathcal B(C(\overline{\mathbb D}),\mathcal B(\LS))
	\longrightarrow
	\mathcal B(\LS) \quad \text{given by} \quad 
	E_{h_{ij}}(T)=T(h_{ij}),
	\]
	is weak-$*$ continuous, we have that $
	\xi_\beta(s^{(i)},s^{(j)})(h_{ij})
	\xrightarrow{w^*}
	\xi(s^{(i)},s^{(j)})(h_{ij})
	$
	in $\mathcal B(\LS)$. On the other hand,
	$
	\xi_\beta(s^{(i)},s^{(j)})(h_{ij})
	=
	w_\beta(s^{(i)},s^{(j)})
	\xrightarrow{w^*}
	w_{ij}
	$
	in $\mathcal{B}(\LS)$. By uniqueness of weak-$*$ limits in $B(\LS)$, $w_{ij}
	=
	\xi(s^{(i)},s^{(j)})
	\left(
	1-\mathrm J(s^{(i)})
	\overline{\mathrm J(s^{(j)})}
	\right)$.
	 Thus, $w\in\mathfrak W$ and so, $\mathfrak W$ is weak-$*$ closed.
	
	\smallskip 
	
	Let $\mu$ be a probability measure on $\DC$ and let $\eta: \G_d \to \LS$ be a function. For $x, y \in \LS$, we denote by $T_{x, y}$ the operator on $\LS$ given by $T_{x, y}(h)=\la h, y \ra x$.  Consider the map $f_1, f_2: \DC \times \G_d \times \G_d \to \mathcal B(\LS)$ defined as
	\[
	f_1(\al, s, t)=\frac{1}{1-\Phi_\al(s)\overline{\Phi_\al(t)}}I_\LS \quad \text{and} \quad 	f_2(\al, s, t)=\frac{1}{1-\Phi_\al(s)\overline{\Phi_\al(t)}}T_{\eta(s), \eta(t)}. 
	\]
	The kernel $k_\alpha(s,t)=1\slash (1-\Phi_\alpha(s)\overline{\Phi_\alpha(t)})$ is a weak kernel being the pullback of the Szeg\H{o} kernel
	under the map $\Phi_\alpha$. Also, the map $(s,t)\mapsto T_{\eta(s),\eta(t)}$ is a weak kernel. Consequently, the product kernel $f_2(\alpha,\cdot,\cdot)$ is a weak kernel. Define $J_{\mu, f_1}, J_{\mu, f_2}: \G_d \times \G_d \to \mathcal{B}(C(\DC), \mathcal{B}(\LS))$  as
	\[
	J_{\mu, f_1}(s, t)(h)=\int_{\DC}h(\alpha)f_1(\alpha, s, t)d\mu(\alpha) \quad \text{and} \quad 	J_{\mu, f_2}(s, t)(h)=\int_{\DC}h(\alpha)f_2(\alpha, s, t)d\mu(\alpha).
	\]
	Clearly, $J_{\mu, f_1}$ and $J_{\mu, f_2}$ are completely positive functions on $\G_d$ such that
	\[
	J_{\mu, f_1}(s, t)(1-\mathrm{J}(s)\overline{\mathrm{J}(t)})=I_\LS \quad \text{and} \quad J_{\mu, f_2}(s, t)(1-\mathrm{J}(s)\overline{\mathrm{J}(t)})=T_{\eta(s), \eta(t)}
	\]
	for every $s, t \in \G_d$. Consequently, the block operator matrix with each entry being $I_\LS$ is in $\mathfrak{W}$ and if $x_1, \dotsc, x_n \in \LS$, then $[T_{x_i, x_j}]_{i, j=1}^n \in \mathfrak{W}$. By applying the Hahn-Banach separation argument and the techniques of Lemma 5.1 of \cite{Tirtha_Sau_II}, the proof of which carries over verbatim
	in the present setting, it follows that the restriction $g_{F_n}$ of $g$ to
	$F_n\times F_n$ belongs to $\mathfrak W$. Since a representation on a larger set restricts naturally to any smaller subset, an application of the same reasoning as in Theorem 11.5 of \cite{Agler_McCarthy}, together with Kurosh’s theorem (see \cite{V_Arkh}, Page 7) gives the following result.

	\begin{lem}\label{lem:decomposition}
		For $F \subseteq \G_d$, let $g: F \times F \mapsto \mathcal B(\LS)$ be a self-adjoint function, that is, $g(s,t)=g(t,s)^*$ for every  $s, t \in F$. Suppose the map $g \oslash k: F \times F \to \mathcal{B}(\LS \otimes \LS)$ given by 
		\[
		g\oslash k(s,t)= g(s,t)\otimes k(s,t)
		\]
		is positive semi-definite for every $\mathcal B(\LS)$-valued admissible weak kernel
		$k$ on $F$. Then there exists a completely positive map $\xi:F \times F \to 
		\mathcal B(C(\overline{\mathbb D}),\mathcal B(\LS))$ such that for every $s, t \in F$,
		\[
		g(s,t)= \xi(s,t)
		\left(
		1-\mathrm{J}(s)\overline{\mathrm{J}(t)}
		\right).
		\]
	\end{lem}

	We recall the notion of a unitary colligation from \cite{Ball, Tirtha_Sau_II} and consider an analogous notion in the present framework. For Hilbert spaces $\LS_1, \LS_2$, a function $f: \G_d \to \mathcal{B}(\LS_1, \LS_2)$ is said to be associated to a \textit{unitary colligation} if there exist a Hilbert space $\HS$, a unital $*$-representation $\rho: C(\DC) \to \mathcal{B}(\HS)$ and a unitary $V: \LS_1 \oplus \HS \to \LS_2 \oplus \HS$ such that for every $s \in \G_d$,
	\[
	f(s)=A+B\rho(\mathrm{J}(s))(I_\HS-D\rho(\mathrm{J}(s)))^{-1}C, \quad 
	\text{where} \quad 
	V=\begin{bmatrix} A & B \\ C & D \end{bmatrix}.		
	\]
	The class of functions $f: \G_d \to \mathcal{B}(\LS_1, \LS_2)$ which is associated to a unitary colligation is denote by $UC_{\G_d}(\LS_1, \LS_2)$. We denote $UC_{\G_d}(\C, \C)$ simply by $UC(\G_d)$.  It follows from the definition of $UC(\G_d)$ that $UC(\G_d) \subseteq S(\G_d)$. Furthermore, $\|\rho(\mathrm{J}(s))\|<1$ since $\rho$ is a unital $*$-representation and $\|\mathrm{J}(s)\|_{\infty, \DC}<1$ for all $s \in \G_d$. Let us recall from \cite{Drit2007_I} that a unital $*$-representation $\rho: C(\DC) \to \mathcal{B}(\HS)$ is said to be \textit{simple} if there exist $\al_1, \dotsc, \al_m \in \DC$ and orthogonal projections $P_1, \dotsc, P_m \in \mathcal{B}(\HS)$ with $P_1+\dotsc +P_m=I_\HS$ such that for all $g \in C(\DC)$,
	\[
	\rho(g)=g(\al_1)P_1+\dotsc+g(\al_m)P_m
	\]
	In the following result, we prove that any $f \in UC(\G_d)$ is in the class $SA(\G_d)$ under the hypothesis that unital $*$-representations in a unitary colligation of $f$ are simple.	
	
	\begin{lem}\label{lem_prelim_III_P}
		Let $f:\mathbb G_d\to \mathcal B(\LS_1,\LS_2)$ be a function in $UC_{\G_d}(\LS_1, \LS_2)$. If $\rho$ is a simple representation, then 
		$f\in SA_{\G_d}(\LS_1,\LS_2)$.
	\end{lem}
	
	\begin{proof}
		Suppose there exist a Hilbert space $\HS$, a unital $*$-representation 
		$
		\rho:C(\mathbb D)\to \mathcal B(\HS)
		$
		and a unitary operator $V: \LS_1 \oplus \HS \to \LS_2 \oplus \HS$ such that for every $s \in \G_d$,
		\[
		f(s)=A+B\rho(\mathrm J(s))(I_{\mathcal H}-D\rho(\mathrm J(s)))^{-1}C, \quad \text{where}\quad 
		V=
		\begin{bmatrix}
			A&B\\
			C&D
		\end{bmatrix}.
		\]
		It is clear that $f\in \operatorname{Hol}(\mathbb G_d, \mathcal B(\LS_1,\LS_2))$. Since $\rho$ is simple, there exist $\alpha_1,\dotsc,\alpha_m \in \DC$ and mutually orthogonal projections
		$P_1,\ldots,P_m\in\mathcal B(\mathcal H)$ such that $P_1+\dotsc+P_m=I_{\HS}$ and
		\[
		\rho(g)=\sum_{j=1}^{m}g(\alpha_j)P_j
		\]
		for every $g \in C(\DC)$. We have by \eqref{eqn_J(z)} that $\rho(\mathrm{J}(s))= \overset{m}{\underset{j=1}{\sum}}P_j \Phi_{\al_j}(s)$. Let $\underline{S}=(S_1, \dotsc, S_{d-1}, P) \in \mathfrak{M}_{\G_d}$ be acting on a Hilbert space $\mathcal{K}$. Then $\rho(J(\underline{S}))=\overset{m}{\underset{j=1}{\sum}}P_j \otimes \Phi_{\al_j}(\underline{S})$. For $h \in \HS$ and $x \in \mathcal{K}$, we have
		\begin{align*}
			\|\rho(J(\underline{S}))(h\otimes x)\|^2
			=
			\left\|
			\sum_{j=1}^{m}P_jh\otimes\Phi_{\alpha_j}(\underline{S})x
			\right\|^2
			&=
			\sum_{j=1}^{m}\|P_jh\|^2
			\|\Phi_{\alpha_j}(\underline{S})x\|^2\\
			&\leq
			\left(\max_{1\leq j\leq m}\|\Phi_{\alpha_j}(\underline{S})\|^2\right)
			\sum_{j=1}^{m}\|P_jh\|^2\|x\|^2\\
			&=
			\left(\max_{1\leq j\leq m}\|\Phi_{\alpha_j}(\underline{S})\|^2\right)
			\|h\|^2\|x\|^2.
		\end{align*}
		Since $\underline{S}\in\mathfrak{M}_{\G_d}$, we have that
		$\|\Phi_{\alpha_j}(\underline{S})\|<1$ for $1\leq j\leq m$ and so,		$\|\rho(J(\underline{S}))\|<1$. Then the operator $I_{\HS\otimes\mathcal K}-(D\otimes I_{\mathcal K})\rho(J(\underline{S}))$ is invertible. The operator-valued functional calculus gives
		\begin{align*}
			f(\underline S)
			=
			A\otimes I_{\mathcal K}
			+
			(B\otimes I_{\mathcal K})
			\rho(\mathrm J(\underline S))
			\left(
			I_{\mathcal H\otimes\mathcal K}
			-(D\otimes I_{\mathcal K})
			\rho(\mathrm J(\underline S))
			\right)^{-1}
			(C\otimes I_{\mathcal K}).
		\end{align*}
		Since $1-\overline{f(s)}f(s)=C^*(I-D\rho(\mathrm{J}(s)))^{-*}(I-\rho(\mathrm{J}(s))^*\rho(\mathrm{J}(s)))(I-D\rho(\mathrm{J}(s)))^{-1}C$ for every $s \in \G_d$, it follows that		
		\begin{align*}
			I_{\LS_1\otimes\mathcal K}
			-f(\underline S)^*f(\underline S)
			=
			R(\underline S)^*
			\left(
			I_{\mathcal H\otimes\mathcal K}
			-\rho(\mathrm J(\underline S))^*
			\rho(\mathrm J(\underline S))
			\right)
			R(\underline S),
		\end{align*}
		where
		$
		R(\underline S)
		=
		\left(
		I_{\mathcal H\otimes\mathcal K}
		-(D\otimes I_{\mathcal K})
		\rho(\mathrm J(\underline S))
		\right)^{-1}
		(C\otimes I_{\mathcal K}).
		$
		Hence, $I_{\LS_1\otimes\mathcal K}
		-f(\underline S)^*f(\underline S)\geq0$ since $\|\rho(\mathrm J(\underline S))\|<1$. Hence,
		$\|f(\underline S)\|\leq1$ and so, $f \in SA_{\G_d}(\LS_1, \LS_2)$.
	\end{proof}
	
Let $k: \G_d \times \G_d \to \mathcal{B}(\LS_2)$ an admissible kernel on $\G_d$, and let $\mathcal H(k)$ denote the reproducing kernel Hilbert space associated with the
$\mathcal B(\LS_2)$-valued kernel $k$, that is,
\[
\mathcal H(k)
=
\overline{\operatorname{span}}
\left\{
k(\cdot,s)u:
s\in\G_d,\;
u\in\LS_2
\right\},
\] 
where $k(\cdot,s)u: \G_d \to \LS_2$ is defined as $t \mapsto k(t,s)u$ and the inner product is given by
\[
\left\langle
k(\cdot,t)v,\,
k(\cdot,s)u
\right\rangle_{\mathcal H(k)}
=
\left\langle
k(s,t)v,\,
u
\right\rangle_{\LS_2}.
\]
Suppose $M_{s_1},\dotsc,M_{s_{d-1}},M_p$ denote the multiplication operators by the
coordinate functions on $\mathcal H(k)$. Since $k$ is admissible, the map $
(s,t)\mapsto
\bigl(1-\Phi_\alpha(s)\overline{\Phi_\alpha(t)}\bigr)k(s,t)
$
is positive semi-definite on $\mathbb G_d\times\mathbb G_d$ for every
$\alpha\in\mathbb D$. For
$
f=k(\cdot,s_1)u_1+\dotsc+k(\cdot, s_n)u_n
$
be a finite linear combination of kernel functions. Then
\begin{align*}
	\|f\|_{\mathcal H(k)}^2-
	\|M_{\Phi_\alpha}f\|_{\mathcal H(k)}^2
	=
	\sum_{i,j=1}^n
	\left\langle
	\bigl(1-\Phi_\alpha(s_i)\overline{\Phi_\alpha(s_j)}\bigr)
	k(s_i,s_j)u_j,
	u_i
	\right\rangle_{\LS_2} \geq 0
\end{align*}
and so, $M_{\Phi_\alpha}$ is a well-defined contraction on
$\mathcal H(k)$. Our next result shows that the coordinate multiplication linear maps $M_{s_1}, \dotsc, M_{s_{d-1}}, M_p$ are bounded on $\HS(k)$. In the case $d=2$, this conclusion is contained in Lemma 3.2 of \cite{Tirtha_Sau}. The argument presented here is different in nature and applies uniformly to all $d \geq 2$.

\begin{prop}\label{prop_Ms}
	Let $k:\mathbb G_d\times\mathbb G_d\to \mathcal B(\mathcal L_2)$ be an admissible kernel, and let
	$\mathcal H(k)$ be the reproducing kernel Hilbert space associated with $k$.
	Then the coordinate multiplication linear maps
	\[
	M_{s_1},\ldots,M_{s_{d-1}},M_p
	\]
	are bounded on $\mathcal H(k)$.
\end{prop}

\begin{proof}
	Since $k$ is admissible	$(s, t) \mapsto (1-\Phi_\alpha(s)\overline{\Phi_\alpha(t)})k(s,t)$ is a positive semi-definite map on $\G_d \times \G_d$ for every $\alpha \in \DC$. By the multiplier criterion for reproducing	kernel Hilbert spaces, $\|M_{\Phi_\alpha}\|\leq 1$ for every $\alpha \in \DC$. Consider the map $F: \DC \to \mathcal{B}(\HS(k))$ given by
	\[
	F(\alpha)=M_{\Phi_\alpha}.
	\]
Note that $\|F(\alpha)\| \leq 1$. We show that $F$ is holomorphic. Since $\Phi_\alpha(s)$ is holomorphic in $\alpha$ for every
	fixed $s\in\mathbb G_d$ and $\HS(k)$ is a separable Hilbert space, it is suffices to show that $\alpha \to \la F(\alpha)f, g \ra_{\HS(k)}$ is holomorphic on $\DC$ for every $f, g \in \HS(k)$. Let $f=k(\cdot,s)v$ and $g=k(\cdot,t)u$, where $s,t \in \mathbb G_d$ and $u, v \in \mathcal L_2$. Then
	\begin{align*}
		\left\langle F(\alpha)f,g\right\rangle_{\mathcal H(k)}
		=
		\left\langle M_{\Phi_\alpha}k(\cdot,s)v,
		k(\cdot,t)u\right\rangle_{\mathcal H(k)}
		=
		\left\langle
		\left(M_{\Phi_\alpha}k(\cdot,s)v\right)(t),
		u
		\right\rangle_{\mathcal L_2}
		=
		\left\langle
		\Phi_\alpha(t)k(t,s)v,
		u
		\right\rangle_{\mathcal L_2}.
	\end{align*}
For fixed $s, t \in \G_d$ and $u,v \in \LS_2$, the map
	$
	\alpha\mapsto
	\left\langle
	\Phi_\alpha(t)k(t,s)v,
	u
	\right\rangle_{\mathcal L_2}
	$
	is holomorphic on $\DC$ since $\alpha\mapsto\Phi_\alpha(t)$ is holomorphic on
	$\DC$. Consequently, $F$ is holomorphic. By the
	Cauchy integral formula for Banach-valued holomorphic functions, the $n$-th derivative
	$F^{(n)}(0)$ evaluated at $\alpha=0$ belongs to $\mathcal B(\mathcal H(k))$ for every $n \in \mathbb{N}$. Since $F(\alpha)=M_{\Phi_\alpha}$ is a $\mathcal B(\mathcal H(k))$-valued holomorphic map, its derivatives are obtained by differentiating the corresponding multipliers. In particular,  
	\[
	F^{(n)}(0)=M_{\left.\frac{\partial^n}{\partial \alpha^n}\Phi_\alpha\right|_{\alpha=0}}
	\]
	for every $n\in\mathbb N$.  We first prove the desired conclusion for $d=2, 3$. Let $d=2$. In this case,
	\[
	\Phi_\alpha(s_1, p)
	=
	\frac{2p\alpha-s_1}{2-s_1\alpha} \quad (s_1, p) \in \G_2). 
	\]
	Then $\Phi_0(s_1, p)=-s_1\slash 2$ and so, $M_{s_1}=-2F(0)$. Thus, $M_{s_1}$ is bounded. Also, $\Phi_\alpha(s_1, p)(2-s_1\alpha)=2p\alpha-s_1$
	and differentiating both sides with respect to $\alpha$ gives that $\Phi'_\alpha(s_1, p)(2-s_1\alpha)-s_1\Phi_\alpha(s_1, p)=2p$. Putting $\alpha=0$, we have that $2M_p=2F'(0)-M_{s_1}F(0)$ and so, $M_p$ is bounded. Thus, the desired conclusion holds for $d=2$. For $d=3$, we have
	\[
	\Phi_\alpha(s_1, s_2, p)
	=
	\frac{-3p\alpha^2+2s_2\alpha-s_1}
	{3-2s_1\alpha+s_2\alpha^2} \quad (s_1, s_2, p) \in \G_3.
	\]
	Then $\Phi_0(s_1, s_2, p)-s_1\slash 3$ and thus $M_{s_1}=-3F(0)$ is bounded. Differentiating both the sides in $\Phi_\alpha(s_1, s_2, p)(3-2s_1\alpha+s_2\alpha^2)=-3p\alpha^2+2s_2\alpha-s_1$ with respect to $\alpha$, it follows that 
	\[
	\Phi'_\alpha(s_1, s_2, p)(3-2s_1\alpha+s_2\alpha^2)
	+\Phi_\alpha(s_1, s_2, p)(-2s_1+2s_2\alpha)
	=-6p\alpha+2s_2.
	\] 
	Putting $\alpha=0$, we have $3F'(0)-2M_{s_1}F(0)=2M_{s_2}$ and so, $M_{s_2}=(3 \slash 2)F'(0)+(1\slash 3)M_{s_1}^2$. Hence, $M_{s_2}$ is bounded. Also, 
	\[
	\Phi''_\alpha(s_1, s_2, p)(3-2s_1\alpha+s_2\alpha^2)
	+
	2\Phi'_\alpha(s_1, s_2, p)(-2s_1+2s_2\alpha)
	+
	2s_2\Phi_\alpha(s_1, s_2, p)
	=
	-6p.
	\]
	Again putting $\alpha=0$, we have that $3F''(0)-4M_{s_1}F'(0)+2M_{s_2}F(0)=-6M_p$ and so, $M_p$ is bounded. Hence, the desired conclusion holds for $d=3$. The general case follows by repeating the above arguments and successively differentiating the identity
	$
	\Phi_\alpha(s)R(s)(\alpha)=Q(s)(\alpha)
	$
	at $\alpha=0$ to recover the coordinate functions recursively.
\end{proof}	
	
	We are now in a position to prove the realization theorem for functions in $SA_{\G_d}(\LS_1, \LS_2)$.

	\begin{thm}\label{thm_realization_Gd}
		Let $\LS_1$ and  $\LS_2$ be Hilbert spaces. For a function $f: \G_d \to \mathcal{B}(\LS_1, \LS_2)$, the following statements are equivalent: 
		\begin{enumerate}[leftmargin=*]
			\item[$(1)$] $f \in SA_{\G_d}(\LS_1, \LS_2)$;
			
			\item[$(2)$] $(s, t) \mapsto (I_{\LS_2}-f(s)f(t)^*)\oslash k(s, t)$ is a positive semi-definite function for all $\mathcal{B}(\LS_2)$-valued admissible kernel $k$ on $\G_d$;
			
			\item[$(3)$] there exists a completely positive map  $\xi: \G_d \times \mathbb{G}_d \to \mathcal{B}(C(\DC), \mathcal{B}(\LS_2))$ such that
			\[
			I_{\LS_2}-f(s)f(t)^*=\xi(s, t)(1-\mathrm{J}(s)\overline{\mathrm{J}(t)})
			\]
			for all $s, t \in \G_d$;
			
			\item[$(4)$] $f \in UC{\G_d}(\LS_1, \LS_2)$.
		\end{enumerate}
	\end{thm}
	
	\begin{proof} $(1)\Longrightarrow (2)$.
		Let $k: \G_d \times \G_d \to \mathcal{B}(\LS_2)$ an admissible kernel on $\G_d$. We first prove the implication for $f\in SA_{\G_d}(\LS_1,\LS_2)\cap
		\operatorname{Hol}(\Gamma_d,\mathcal B(\LS_1,\LS_2))$. The general case follows by approximation arguments. Let $\mathcal H(k)$ denote the reproducing kernel Hilbert space associated with the
		$\mathcal B(\LS_2)$-valued kernel $k$, that is,
		$\mathcal H(k)
		=
		\overline{\operatorname{span}}
		\left\{
		k(\cdot,s)u:
		s\in\G_d,\;
		u\in\LS_2
		\right\}$, where $k(\cdot,s)u: \G_d \to \LS_2$ is defined as $t \mapsto k(t,s)u$ and the inner product is given by
		\[
		\left\langle
		k(\cdot,t)v,\,
		k(\cdot,s)u
		\right\rangle_{\mathcal H(k)}
		=
		\left\langle
		k(s,t)v,\,
		u
		\right\rangle_{\LS_2}.
		\]
		Suppose $M_{s_1},\dotsc,M_{s_{d-1}},M_p$ are the multiplication operators by the
		coordinate functions on $\mathcal H(k)$. We have by Proposition \ref{prop_Ms} that $M_{s_1}, \dotsc, M_{s_{d-1}}, M_p$ are bounded, and
		$
		\|\Phi_\alpha(M_{s_1},\ldots,M_p)\|\leq 1
		$
		for every $\alpha \in \DC$. Fix $F_m=\{s^{(1)},\ldots,s^{(m)}\}\subseteq\G_d$ with $s^{(\ell)}=(s_1^{(\ell)},\ldots,s_{d-1}^{(\ell)},p^{(\ell)})$.
		Let
		\[
		\mathcal H_m(k)
		=
		\operatorname{span}
		\{
		k(\cdot,s^{(\ell)})u:
		u\in\LS_2,\,
		1\le\ell\le m
		\}.
		\]
		Since $M_{s_i}^*(k(\cdot,s^{(\ell)})u)
		=
		\overline{s}_i^{(\ell)}\,k(\cdot,s^{(\ell)})u$ for  $1\leq i\leq d-1$ and
		$
		M_p^*(k(\cdot,s^{(\ell)})u)
		=
		\overline{p}^{(\ell)}\,k(\cdot,s^{(\ell)})u$, the space $\mathcal H_m(k)$ is jointly invariant under
		$(M_{s_1}^*,\ldots,M_{s_{d-1}}^*,M_p^*)$. Let
		\[
		(S_1^*,\ldots,S_{d-1}^*,P^*)
		=
		\left(M_{s_1}^*|_{\mathcal H_m(k)},\ldots,
		M_{s_{d-1}}^*|_{\mathcal H_m(k)},
		M_p^*|_{\mathcal H_m(k)}\right).
		\]
		Then $\|\Phi_\alpha(S_1^*,\ldots,S_{d-1}^*,P^*)\|
		\leq 1$ for every $\alpha\in\overline{\mathbb D}$. Since $\G_d$ is $(1,\dotsc, d-1, d)$-quasi-balanced, $r\cdot\underline S^*
		=
		(rS_1^*,\ldots,r^{d-1}S_{d-1}^*,r^dP^*)
		\in\mathfrak M_{\G_d}$ for $0<r<1$. Consider a holomorphic function $\widehat{f}: \Gamma_d \to \mathcal{B}(\LS_2, \LS_1)$ given by $\widehat f(s)=f(\overline s)^*$. Evidently, $\widehat{f} \in  SA_{\G_d}(\LS_2,\LS_1)$.  Since $\Gamma_d$ is a polynomially convex compact set, there exists a sequence of
		$\mathcal B(\LS_2,\LS_1)$-valued polynomials
		$\{p_n\}_{n\ge1}$ converging uniformly to $\widehat f$ on $\Gamma_d$ (see Theorem 28.3 in \cite{Mujica}). By holomorphic functional calculus, $\widehat{f}(r.\underline{S}^*) \in \mathcal{B}(\HS_m \otimes \LS_2, \HS_m \otimes \LS_1)$. By the reproducing property and polynomial approximation, we have that for every $u, v \in \LS_2$ and $1 \leq \ell \leq m$,
		\[
		\widehat f(r\cdot\underline S^*)\left(k(\cdot, s^{(\ell)})u \otimes v\right)
		=
		k(\cdot, s^{(\ell)})u \otimes  f(r\cdot s)^*v
		\] 
		Since $r\cdot\underline S^*\in\mathfrak M_{\G_d}$, it follows that 
		$\|\widehat f(r\cdot\underline S^*)\|
		\leq 1$. For $\{u_i, v_i: i \leq i \leq n\} \subseteq \LS_2$, we have 
		\begin{align*}
			0 & \leq \left\la \left(I_{\HS_m(k)\otimes \LS_2}-\widehat{f}(r\cdot \underline{S}^*)^*\widehat{f}(r \cdot \underline{S}^*)\right) \left[\overset{n}{\underset{j=1}{\sum}}k(\cdot, s^{(j)})u_j \otimes v_j \right], \left[\overset{n}{\underset{i=1}{\sum}}k(\cdot, s^{(i)})u_i \otimes v_i \right]\right\ra \\
			&= \overset{n}{\underset{i, j=1}{\sum}} \left \la k(s^{(i)}, s^{(j)})u_j, u_i \right \ra
			\left \la \left(I_{\LS_2}-f(r\cdot s^{(i)}) f(r\cdot s^{(j)})^*\right)v_j,v_i \right \ra \\
			&=\sum_{i,j=1}^{n}
			\left\langle
			\left(
			(I_{\LS_2}-f(r\cdot s^{(i)})
			f(r\cdot s^{(j)})^*)
			\otimes
			k(s^{(i)},s^{(j)})
			\right)
			(v_j\otimes u_j),
			v_i\otimes u_i
			\right\rangle.
		\end{align*}
		Thus, the map $(s,t) \mapsto
		(I_{\LS_2}-f(r\cdot s)f(r\cdot t)^*)
		\oslash k(s,t)$ is positive semi-definite on \(\G_d\times\G_d\). For a general $f \in SA_{\G_d}(\LS_1,\LS_2)$,
		define
		$f_r(s)=f(r\cdot s)$.
		Then
		$f_r\in
		SA_{\G_d}(\LS_1,\LS_2)\cap
		\operatorname{Hol}(\Gamma_d,\mathcal B(\LS_1,\LS_2))$.
		Applying the previous argument to $f_r$ and letting
		$r\uparrow1$ gives the desired conclusion.
		
		\medskip 
		
		\noindent $(2)\implies (3)$.  Suppose $(s, t) \mapsto (I_{\LS_2}-f(s)f(t)^*)\oslash k(s, t)$ is a positive semi-definite function for all $\mathcal{B}(\LS_2)$-valued admissible kernel $k$ on $\G_d$. Consider the continuous map 
		$g: \G_d \times \G_d \to \mathcal{B}(\LS_2)$ defined as $g(s, t)=I_{\LS_2}-f(s)f(t)^*$. Evidently, $g$ is a self-adjoint map such that $(s, t) \mapsto g(s, t)\otimes k(s, t)$ is positive semi-definite for all $\mathcal{B}(\LS_2)$-valued admissible weak kernel $k$.
		For $\varepsilon>0$, define
		\[
		k_\varepsilon(s,t)=k(s,t)+\varepsilon\delta(s,t)I_{\LS_2},
		\]
		where $\delta(s,t)$ denotes the Kronecker delta function on $\G_d\times\G_d$. Then $k_\varepsilon$ is a $\mathcal{B}(\LS_2)$-valued admissible weak kernel on $\G_d$. Hence,
		$
		(s,t)\mapsto (1-f(s)f(t)^*)k_\varepsilon(s,t)
		$
		is positive semi-definite for every $\varepsilon>0$. Letting $\varepsilon\to0$, we conclude that
		$
		(s,t)\mapsto (1-f(s)f(t)^*)k(s,t)
		$
		is positive semi-definite for every admissible weak kernel $k$ on $\G_d$. The desired conclusion now follows from Lemma \ref{lem:decomposition}.
		
		\medskip 
		
		\noindent $(3) \implies (4)$. Suppose there exist a completely positive map  $\xi: \G_d \times \mathbb{G}_d \to \mathcal{B}(C(\DC), \mathcal{B}(\LS_2))$ such that
		$
		I_{\LS_2}-f(s)f(t)^*=\xi(s, t)(1-\mathrm{J}(s)\overline{\mathrm{J}(t)})
		$
		for all $s, t \in \G_d$ By Proposition \ref{prop_204}, there exist a Hilbert space $\HS$, a map $L: \G_d \to \mathcal{B}(C(\DC), \mathcal{B}(\HS, \LS_2))$ and a unital $*$-representation $\rho: C(\DC) \to \mathcal{B}(\HS)$ such that $\xi(s, t)(h\overline{g})=L(s)(h)(L(t)(g))^*$ and $L(s)(hg)^*=\rho(h)^*L(s)(g)^*$ for all $h, g \in C(\DC)$ and $s, t \in \G_d$. Hence, $I_{\LS_2}-f(s)f(t)^*
		=(L(s)1)(L(t)1)^*-(L(s)\mathrm{J}(s))(L(t)\mathrm{J}(t))^*$ for all $s, t \in \G_d$. Then
		\begin{align*}
			I_{\LS_2}+L(s)(1)\rho(\mathrm{J}(s))\rho(\mathrm{J}(t))^*(L(t)(1))^* =f(s)f(t)^*+ L(s)(1)(L(t)(1))^*
		\end{align*}
		for all $s, t \in \G_d$. Consequently, the linear map 
		\[
		V: \overline{\text{span}} \left\{
		\begin{bmatrix}
			u \\
			\rho(\mathrm{J}(t))^*(L(t)(1))^*u
		\end{bmatrix}
		: u \in \LS_2,  t \in \G_d
		\right\}\to  \overline{\text{span}} \left\{
		\begin{bmatrix}
			f(t)^*u \\
			(L(t)(1))^*u
		\end{bmatrix}
		: u \in \LS_2, t \in \G_d
		\right\}
		\]
		given by 
		\[   \begin{bmatrix}
			u \\
			\rho(\mathrm{J}(t))^*(L(t)(1))^*u
		\end{bmatrix}\mapsto \begin{bmatrix}
			f(t)^*u \\
			(L(t)(1))^*u
		\end{bmatrix}
		\]
		is an isometry $V$ and thus, it can be extended to a unitary from $\LS_2 \oplus \HS$ onto $\LS_1 \oplus \HS$. We can write $V=\begin{bmatrix} A & B \\ C & D \end{bmatrix}$. For every $t \in \G_d$ and $u \in \LS_2$, $Au+B\rho(\mathrm{J}(t))^*(L(t)(1))^*u=f(t)^*u$ and $Cu+D\rho(\mathrm{J}(t))^*(L(t)(1))^*u=(L(t)(1))^*u$. Since $\|\rho(\mathrm{J}(t))\|<1$ and $\|D\| \leq 1$, we have that $(L(t)(1))^*u=(I_\HS-D\rho(\mathrm{J}(t))^*)^{-1}Cu$ for every $u \in \LS_2$. Then 
		$
		f(t)^*=A+B\rho(\mathrm{J}(t))^*(I_\HS-D\rho(\mathrm{J}(t))^*)^{-1}C
		$ 
		and
		\[
		f(t)=A^*+C^*(I_{\HS}-\rho(J(t))D^*)^{-1}\rho(J(t))B^*=A^*+C^*\rho(J(t))(I_{\HS}-D^*\rho(J(t)))^{-1}B^*,
		\]
		which corresponds to the unitary $V^*: \LS_1\oplus \HS \to \LS_2 \oplus \HS$. Therefore, $f \in UC{\G_d}(\LS_1, \LS_2)$.
		
		\medskip 
		
		\noindent $(4) \implies (1)$. Suppose 
		$
		f(s)=A+B\rho(\mathrm J(s))(I_{\HS}-D\rho(\mathrm J(s)))^{-1}C,
		$
		where
		\[
		V=
		\begin{bmatrix}
			A&B\\
			C&D
		\end{bmatrix}:
		\LS_1\oplus\HS
		\longrightarrow
		\LS_2\oplus\HS
		\]
		is unitary and
		\(
		\rho:C(\DC)\to\mathcal B(\HS)
		\)
		is a unital $*$-representation. By the representation theorem for spectral measures, there exists a unique $\mathcal B(\HS)$-valued spectral measure $\mu$ on the Borel
		$\sigma$-algebra of $\DC$ such that
		\[
		\rho(h)=\int_{\DC}h(\alpha)\,d\mu(\alpha),
		\qquad h\in C(\DC).
		\]
		Consider a collection $\mathfrak{F}$ consisting of pairs $\beta=(F, \epsilon)$, where $F$ is a finite subset of $\G_d$ and $\epsilon >0$ ordered by $(F_1, \epsilon_1)\leq (F_2, \epsilon_2)$ if $F_1 \subseteq F_2$ and $\epsilon_1 \geq \epsilon_2$. This makes $\mathfrak{F}$ a directed set. Let $\beta=(F, \epsilon) \in \mathfrak{F}$ and let us consider the collection given by
		\[
		\Lambda=\left\{\Phi_{\alpha} : \al \in \DC \right\}.
		\]
		Clearly, $\Lambda \subseteq B(\G_d, \DC)$ (or equivalently, $\DC^{\G_d})$, the collection of bounded functions from $\G_d$ into $\DC$, which is endowed with the topology of pointwise convergence. By Tychonov's theorem, $B(\G_d, \DC)$ is a compact Hausdorff space. Now $\DC$ is a Tychonov space in the usual metric topology (that is, points are closed and for any closed set and point
		disjoint from it, there is a continuous bounded function separating the two). Consequently, $\DC^{\G_d}$ is Tychonov and so, $\Lambda$ is a Tychonov space. By compactness of $\Lambda$, there exists a finite open cover $\mathcal{U}=\{U^{\beta}_1, \dots, U^{\beta}_m\}$ of $\Lambda$ with the property that $|\psi'(s) - \psi''(s)| < \varepsilon$ for every $\psi', \psi'' \in U^{\beta}_j$ for $1 \leq j \leq m$ and $s \in F$. Let us form a partition $\{\Delta_1^\beta, \dotsc, \Delta_m^\beta\}$ of $\Lambda$ from $\mathcal{U}$ as follows:
		\[
		\Delta^{\beta}_1 = U^{\beta}_1, \quad
		\Delta^{\beta}_2 = U^{\beta}_2 \setminus U^{\beta}_1, \quad 
		\dots, \quad 
		\Delta^{\beta}_m = U^{\beta}_m \setminus (U_1^{\beta} \cup \dotsc \cup U_{m-1}^\beta).
		\]
		We now have a cover  $\{\Delta^{\beta}_1, \dotsc, \Delta^{\beta}_m\}$ of $\Lambda$ consisting of mutually disjoint Borel subsets such that 
		\begin{align}\label{eqn_RP_007}
			|\psi'(s)-\psi''(s)|<\epsilon \quad \text{for every} \quad \psi', \psi'' \in \Delta_j^\beta, \ s \in F
		\end{align}
		for $1 \leq j \leq m$. Also, consider collections $\{\psi_1^\beta, \dotsc, \psi_m^\beta\}$ such that $\psi_j^\beta=\psi_{\theta_j} \in \Delta_j^\beta$ for some $\theta_j \in \DC$, where $j=1, \dotsc, m$. Let us define the map $\eta: \DC \to \Lambda$  and $\rho_{\beta}: C(\DC) \to \mathcal{B}(\HS)$ as follows:
		\[
		\eta(\alpha)=\Phi_\alpha \quad \text{and} \quad  \rho_{\beta}(h)=\overset{m}{\underset{j=1}{\sum}}\mu\left(\eta^{-1}(\Delta_j^\beta)\right)h(\theta_j).
		\]
		Since $\{\Delta_1^\beta, \dotsc, \Delta_m^\beta\}$ is a partition of $\Lambda$ and $\eta(\DC)=\Lambda$, it follows that $\{\eta^{-1}(\Delta_1^\beta), \dotsc, \eta^{-1}(\Delta_m^\beta)\}$ is a partition of $\DC$. Consequently, the operators $\mu\left(\eta^{-1}(\Delta_j^\beta)\right)$ are pairwise orthogonal projections for $1 \leq j \leq m$ such that $\overset{m}{\underset{j=1}{\sum}}\mu\left(\eta^{-1}(\Delta_j^\beta)\right)=I_{\HS}$ and so, $\rho_{\beta}$ is a simple representation and by \eqref{eqn_RP_007}, 
		$\|\rho_{\beta}(\mathrm{J}(s))-\rho(\mathrm{J}(s))\| \leq \epsilon$ for $s \in F$. Consider the map $f_\beta: \G_d \to \mathcal{B}(\LS_1, \LS_2)$ given by
		\[
		f_\beta(s)=A+B\rho_\beta(\mathrm{J}(s))(I_\HS-D \rho_\beta(\mathrm{J}(s)))^{-1}C.
		\] 
		By Lemma \ref{lem_prelim_III_P}, each $f_\beta \in SA_{\G_d}(\LS_1, \LS_2)$. Next,  we show that the net $\{f_\beta\}$ converges pointwise to $f$ on $\G_d$. Let $s \in \G_d$ and $\epsilon'>0$. Choose a finite set $F$ in $\G_d$ such that $s \in F$ and consider $\beta'=(F, \epsilon') \in \mathfrak{F}$. As $\|\mathrm{J}(t)\|_{\infty, \DC}<1$ for all $t \in \G_d$, we can define positive scalars $\delta, r_1, r_2$ and $\epsilon$ as follows:
		\[
		\delta=\min\left\{\frac{1-\|\mathrm{J}(t)\|_{\infty, \DC}}{2} : t \in F \right\}, \ \ r=1-\frac{\delta}{2} \ \ \text{and} \ \ \epsilon<\min\{\delta\slash 2, \epsilon'\}.
		\]
		Choose $\beta=(F, \epsilon) \in \mathfrak{F}$. By the definition of $\mathfrak{F}$, it follows that $\beta' \leq \beta$. Since $\rho$ is a unital $*$-representation, we have that $\|\rho_\beta(\mathrm{J}(t))-\rho(\mathrm{J}(t))\| \leq \epsilon$ for each $t \in F$ and
		\begin{align*}
			\|\rho_{\beta}(\mathrm{J}(t))\| \leq \|\rho_{ \beta}(\mathrm{J}(t))-\rho(\mathrm{J}(t))\|+\|\rho(\mathrm{J}(t))\| \leq \epsilon +1-2\delta<r.
		\end{align*}
		Since $D$ is a contraction, $\|D\rho_\beta(\mathrm{J}(t))\| \leq r$ for all $t \in F$. Let $t \in F$. Note that 
		\begin{align*}
			(D\rho_\beta(\mathrm{J}(t)))^n-(D\rho(\mathrm{J}(t)))^n
			&=D(\rho_\beta(\mathrm{J}(t))-\rho(\mathrm{J}(t)))(D\rho_\beta(\mathrm{J}(t)))^{n-1}\\
			&\quad +(D\rho(\mathrm{J}(t)))( \rho_\beta(\mathrm{J}(t))-\rho(\mathrm{J}(t))) (D\rho_\beta(\mathrm{J}(t)))^{n-2} \\ 
			& \quad +\dotsc+(D\rho(\mathrm{J}(t)))^{n-1}D(\rho_\beta(\mathrm{J}(t))-\rho(\mathrm{J}(t))) 
		\end{align*}
		and so, $\|(D\rho_\beta(\mathrm J(t)))^n-(D\rho(\mathrm J(t)))^n\| \leq n\epsilon r^{n-1}$. Then 
		{\small
			\begin{align*}
				& \|\rho_\beta(\mathrm J(t))(I_\HS-D\rho_\beta(\mathrm J(t)))^{-1}-\rho(\mathrm J(t))(I_\HS-D\rho(\mathrm J(t))^{-1}\|\\
				& \leq \|\rho_\beta(\mathrm J(t))-\rho(\mathrm J(t)\| \|(I_\HS-D\rho_\beta(\mathrm J(t)))^{-1}\|+\|\rho(\mathrm J(t)\| \|(I_\HS-D\rho_\beta(\mathrm J(t)))^{-1}-(I_\HS-D\rho(\mathrm J(t))^{-1}\|\\
				& \leq \frac{\epsilon}{(1-r)^2}
			\end{align*}
		}
		\par \noindent 	for all $t \in F$. Thus the bounded net $\{f_\beta\}$ converges pointwise to $f$ on $\G_d$. Since the net $\{f_\beta\}$ is uniformly bounded by $1$, we have by Theorem 1.4.31 in \cite{Scheidemann} that there exists a subsequence $\{f_{\beta_\ell}\}$ of the net that converges uniformly over compacts subsets of $\G_d$ to $f$. Let $\underline{T} \in \mathfrak{M}_{\G_d}$, and let $\underline{T}$ be acting on a Hilbert space $\mathcal{K}$. Since $f_{\beta_\ell} \in SA_{\G_d}(\LS_1, \LS_2)$, we have that $f_{\beta_\ell}(\underline{T}): \LS_1 \otimes \KS \to \LS_2 \otimes \KS$ satisfying $\|f_{\beta_\ell}(\underline{T})\| \leq 1$ for each $\ell$. Let $u=x\otimes \zeta \in\LS_1\otimes\mathcal K$ and	$v=y\otimes\eta\in\LS_2\otimes\mathcal K$. Then
		\[
		\langle f_{\beta_\ell}(\underline T)u,v\rangle
		=
		\langle f_{\beta_\ell, x,y}(\underline T)\zeta,\eta\rangle,
		\quad \text{where} \quad  f_{\beta_\ell, x,y}(s)
		=
		\langle f_{\beta_\ell}(s)x,y\rangle_{\LS_2}
		\qquad
		(s \in\G_d).
		\]
		Since $f_{\beta_\ell}\to f$ uniformly on compact subsets of $\G_d$, it follows that
		$
		f_{\beta_\ell, x,y}\longrightarrow f_{x,y}$ uniformly on compact subsets of $\G_d$. By the functional calculus presented in \cite{Vasilescu}, it follows that the sequence $f_{\beta_\ell, x, y}(\underline{T}) \to f_{x, y}(\underline{T})$ and so,
		\[
		\langle f_\beta(\underline T)u,v\rangle
		\longrightarrow
		\langle f(\underline T)u,v\rangle.
		\]
		Since $|\langle f_\beta(\underline T)u,v\rangle|
		\leq
		\|u\|\,\|v\|$, we have that
		$
		|\langle f(\underline T)u,v\rangle|
		\le
		\|u\|\,\|v\|.
		$
		By linearity, the same estimate holds for finite sums of elementary tensors. Using density arguments, we finally have that
		\[
		\|f(\underline T)\|
		=
		\sup_{\|u\|=\|v\|=1}
		|\langle f(\underline T)u,v\rangle|
		\le
		1.
		\]
		Therefore, $f \in SA_{\G_d}(\LS_1, \LS_2)$. The proof is now complete.
	\end{proof} 	
	
	\section{Applications: Interpolation, Toeplitz corona and Extension theorems on $\G_d$}\label{sec_03}
	
	\noindent In this section, we explore applications of the realization theorem for the symmetrized polydisc to several well-known problems in classical function theory, including interpolation, Toeplitz corona and extension theorems. We begin by presenting the following interpolation theorem on $\G_d$.
	
\begin{thm}\label{thm_interpolation_Gd}
	Assume that $\LS_1, \LS_2$ are Hilbert spaces. Let $F=\{s^{(1)}, \dotsc, s^{(n)}\} \subseteq \G_d$, and let $A_1, \dotsc, A_n \in \mathcal{B}(\LS_1, \LS_2)$. Then the following are equivalent:
	\begin{enumerate}[leftmargin=*]
		\item[$(1)$] there exists a function $f \in SA_{\G_d}(\LS_1, \LS_2)$ such that $f(s^{(i)})=A_i$ for $1 \leq i \leq n$;
		\item[$(2)$] $\begin{bmatrix} (I_{\LS_2}-A_iA_j^*)\otimes k(s^{(i)}, s^{(j)})\end{bmatrix}_{i, j=1}^n \geq 0$ for every $\mathcal{B}(\LS_2)$-valued kernel $k$ on $\G_d$;
		\item[$(3)$] there exists a completely positive map  $\xi: \G_d \times \mathbb{G}_d \to \mathcal{B}(C(\DC), \mathcal{B}(\LS_2))$ such that
		\[
		I_{\LS_2}-A_iA_j^*=\xi(s^{(i)}, s^{(j)})(1-\mathrm{J}(s^{(i)})\overline{\mathrm{J}(s^{(j)})})
		\]
		for $1 \leq i, j \leq n$.
	\end{enumerate}
\end{thm}

\begin{proof}
	The implication $(1) \implies (2)$ follows from the equivalence of $(1)$ and $(2)$ in Theorem \ref{thm_realization_Gd}. The part $(2) \implies (3)$ can be obtained by applying Lemma \ref{lem:decomposition} to the self-adjoint map $g : F \times F \to \mathcal{B}(\LS_1, \LS_2)$ defined as $g(s^{(i)}, s^{(j)}) = I_{\LS_2} - A_iA_j^*$. It remains to prove that $(3) \implies (1)$. Suppose there exists a completely positive map  $\xi: \G_d \times \mathbb{G}_d \to \mathcal{B}(C(\DC), \mathcal{B}(\LS_2))$ such that
	\[
	I_{\LS_2}-A_iA_j^*=\xi(s^{(i)}, s^{(j)})(1-\mathrm{J}(s^{(i)})\overline{\mathrm{J}(s^{(j)})})
	\]
	for $1 \leq i, j \leq n$. It follows from Proposition \ref{prop_204} that there exist a Hilbert space $\HS$, a function $L: F \to \mathcal{B}(C(\DC), \mathcal{B}(\HS, \LS_2))$ and a unital $*$-representation $\rho: C(\DC)\to \mathcal{B}(\HS)$ such that
	\[
	\xi(s^{(i)}, s^{(j)})(f\overline{h})=L(s^{(i)})(f)(L(s^{(j)})(h))^* \quad \text{and} \quad L(s^{(i)})(fh)^*=\rho(f)^*(L(s^{(i)})(h))^*
	\] 
	for $1 \leq i, j \leq n$ and $f, h \in C(\DC)$. Consequently, 
	\[
		I_{\LS_2}-A_iA_j^*
		=(L(s^{(i)})1)(L(s^{(j)})1)^*-(L(s^{(i)})\mathrm{J}(s^{(i)}))( L(s^{(j)})\mathrm{J}(s^{(j)}))^* 
		\]
	and so,
	\[
	I_{\LS_2}+L(s^{(i)})(1)\rho(\mathrm{J}(s^{(i)}))\rho(\mathrm{J}(s^{(j)}))^*(L(s^{(j)})(1))^*=A_iA_j^*+L(s^{(i)})(1)(L(s^{(j)})(1))^*
	\]
	for $1 \leq i, j \leq n$. Therefore, the linear map 
	{\small \begin{align*}
		V: \overline{\text{span}} \left\{
		\begin{bmatrix}
			u \\
			\rho(\mathrm{J}(s^{(i)}))^*(L(s^{(i)})1)^*u 	
		\end{bmatrix}
		: u \in \LS_2, 1 \leq i \leq n
		\right\} \to \overline{\text{span}} \left\{
		\begin{bmatrix}
			A_i^*u \\
			(L(s^{(i)})1)^*u
		\end{bmatrix}
		: u \in \LS_2, 1 \leq i \leq n
		\right\}
	\end{align*}
}
\par \noindent 	defined as 	
\[
V\begin{bmatrix}
		u \\
		\rho(\mathrm{J}(s^{(i)}))^*(L(s^{(i)})1)^*u 
	\end{bmatrix}
	= \begin{bmatrix}
		A_i^*u \\
		(L(s^{(i)})1)^*u 
	\end{bmatrix}.
	\] 
One can now extend $V$ to a unitary from $\LS_2 \oplus \HS$ onto $\LS_1 \oplus \HS$. With respect to this decomposition, we write $V=\begin{bmatrix} A & B \\ C & D \end{bmatrix}$. Then $V^*=\begin{bmatrix} A^* & C^* \\ B^* & D^* \end{bmatrix}: \LS_1 \oplus \HS \to \LS_2 \oplus \HS$ is a unitary. It follows from Theorem \ref{thm_realization_Gd} that the map $f: \G_d \to \mathcal{B}(\LS_1, \LS_2)$ given by
	\[
	f(s)=A^*+C^*\rho(\mathrm J(s))\left(I_\HS-D^*\rho(\mathrm J(s))\right)^{-1}B^*
	\]
	is in $SA_{\G_d}(\LS_1, \LS_2)$. Note that for every $s \in \G_d$,  
	\[
	f(s)^*=A+B\left(I_\HS-\rho(\mathrm J(s))^*D\right)^{-1}\rho(\mathrm J(s))^*C=A+B\rho(\mathrm J(s))^*\left(I_\HS-D\rho(\mathrm J(s))^*\right)^{-1}C.
	\]
	Let every $u \in \LS_2$. By definition of the maps $f(s)$ and $V$, we have 
	\begin{align*}
		A_i^*u=Au+B\rho(\mathrm{J}(s^{(i)}))^*(L(s^{(i)})1)^*u=Au+B\rho(\mathrm J(s^{(i)}))^*\left(I_\HS-D\rho(\mathrm J(s^{(i)}))^*\right)^{-1}Cu=f(s^{(i)})^*u
	\end{align*}
	 and so, $f(s^{(i)})=A_i$ for $1 \leq i \leq n$. The proof is now complete. 
\end{proof}
	
Next, we present the Toeplitz corona theorem on the symmetrized polydisc.
	
\begin{thm}\label{thm_TC_Gd}
	Let $\KS_1, \KS_2$ and $\KS_3$ be Hilbert spaces. For maps $\zeta^{(12)}: \G_d \to \mathcal{B}(\KS_1, \KS_2)$ and $\zeta^{(32)}: \G_d \to \mathcal{B}(\KS_3, \KS_2)$, the following are equivalent: 
	\begin{enumerate}[leftmargin=*]
		\item[$(1)$] there exists $\zeta^{(31)} \in SA_{\G_d}(\KS_3, \KS_1)$ such that $\zeta^{(12)}(s)\zeta^{(31)}(s)=\zeta^{(32)}(s)$ for every $s \in \G_d$;
		\item[$(2)$] for every finite subset $F=\{s^{(1)}, \dotsc, s^{(n)}\} \subset \G_d$,
		\[
		\begin{bmatrix}
			\left(\zeta^{(12)}(s^{(i)})\zeta^{(12)}(s^{(j)})^*-\zeta^{(32)}(s^{(i)})\zeta^{(32)}(s^{(j)})^*\right)\otimes k(s^{(i)}, s^{(j)})
		\end{bmatrix}_{i, j=1}^n \geq 0
		\]
		for every $\mathcal{B}(\KS_2)$-valued admissible kernel $k$ on $\G_d$;
		\item[$(3)$] there exist a completely positive map $\xi: \G_d \times \G_d \to \mathcal{B}(C(\DC), \mathcal{B}(\KS_2))$ such that 
		\begin{align*}
			\zeta^{(12)}(s)\zeta^{(12)}(t)^*-\zeta^{(32)}(s)\zeta^{(32)}(t)^*
			=\xi(s, t)(1-\mathrm{J}(s)\overline{\mathrm{J}(t)})
		\end{align*}
		for all $s, t \in \G_d$.	 	
	\end{enumerate}
\end{thm}

\begin{proof}	$(1) \implies (2)$. Let $k$ be a $\mathcal{B}(\KS_2)$-valued admissible kernel on $\G_d$ and $F=\{s^{(1)}, \dotsc, s^{(n)}\} \subseteq \G_d$. Consider the reproducing kernel Hilbert space $\HS(k)$ and its subspace $\HS_n(k)$ given by
	\[
	\HS(k)=\overline{\text{span}}\{k(., s)u: s \in \G_d, u \in \KS_2\} \quad \text{and} \quad \HS_n(k)=\overline{\text{span}}\{k(., s^{(\ell)})u: u \in \KS_2, 1 \leq \ell \leq m\}.
	\] 
Here, $k(., s)u: \G_d \to \KS_2$ is the map defined as $t \mapsto k(t, s)u$ that satisfies $\la k(., t)u, k(., s)v\ra_{\HS(k)}=\la k(t, s)u, v\ra_{\KS_2}$ for every $s, t \in \G_d$ and $u, v \in \KS_2$. Let $\underline{M}=(M_{s_1}, \dotsc, M_{s_{d-1}}, M_{p})$ be the tuple of coordinate multiplication operators on $\HS(k)$. Then $\HS_n(k)$ is a joint invariant subspace for $\underline{M}^*=(M_{s_1}^*, \dotsc, M_{s_{d-1}}^*, M_{p}^*)$ and so, we define $\underline{S}=(S_1, \dotsc, S_{d-1}, P)$ on $\HS_n(k)$ as
	\[
	(S_1^*, \dotsc, S_{d-1}^*, P^*)=\left(M_{s_1}^*|_{\HS_n(k)}, \dotsc, M_{s_{d-1}}^*|_{\HS_n(k)}, M_{p}^*|_{\HS_n(k)}\right).
	\]
Since $\G_d$ is a $(1, \dotsc, d-1, d)$-quasi-balanced domain, $r\cdot\underline{S}^*=(rS_1^*, \dotsc, r^{d-1}S_{d-1}^*, r^dP^*) \in \mathfrak{M}_{\G_d}$ for every $r \in (0, 1)$. Let $f: \Gamma_d \to \mathcal{B}(\KS_3, \KS_1)$ be a holomorphic map such that $f \in SA_{\G_d}(\KS_3, \KS_1)$. Consider the holomorphic function $\widehat{f}: \Gamma_d \to \mathcal{B}(\KS_1, \KS_3)$ defined as $\widehat{f}(s)=f(\overline{s})^*$ for $s \in \Gamma_d$. Then $\widehat{f} \in SA_{\G_d}(\KS_1, \KS_3)$ and $\widehat{f}(r\cdot \underline{S}^*) \in \mathcal{B}(\HS_n(k)\otimes \KS_1, \HS_n(k)\otimes \KS_3)$. Let $\{u_i, v_i : 1 \leq i \leq n\} \subset \KS_2$. Following the proof of the part $(1) \implies (2)$ in Theorem \ref{thm_realization_Gd}, one can easily prove that
{\small 
	\begin{align*}
		\overset{n}{\underset{i, j=1}{\sum}} \left \la k(s^{(i)}, s^{(j)})u_j, u_i \right \ra
		\left \la \left(\zeta^{(12)}(s^{(i)})\zeta^{(12)}(s^{(j)})^*-\zeta^{(12)}(s^{(i)})\zeta^{(31)}(s^{(i)}) \zeta^{(31)}(s^{(j)})^*\zeta^{(12)}(s^{(j)})^*\right)v_j,v_i \right \ra
		\end{align*}
}\par\noindent	is non-negative and so, $\begin{bmatrix}
		\left(\zeta^{(12)}(s^{(i)})\zeta^{(12)}(s^{(j)})^*-\zeta^{(32)}(s^{(i)})\zeta^{(32)}(s^{(j)})^*\right)\otimes k(s^{(i)}, s^{(j)})
	\end{bmatrix}_{i, j=1}^n \geq 0$.

	\medskip

	\noindent	$(2) \implies (3)$. The conclusion follows by applying Lemma \ref{lem:decomposition} to the map $g : \G_d \times \G_d \to \mathcal{B}(\KS_2)$ defined as $g(z, w)=\zeta^{(12)}(z)\zeta^{(12)}(w)^*-\zeta^{(32)}(z)\zeta^{(32)}(w)^*$.

	\medskip

	\noindent	$(3) \implies (1)$. Assume that there is a completely positive map $\xi: \G_d \times \G_d \to \mathcal{B}(C(\DC), \mathcal{B}(\KS_2))$ such that $
		\zeta^{(12)}(s)\zeta^{(12)}(t)^*-\zeta^{(32)}(s)\zeta^{(32)}(t)^* =\xi(s, t)(1-\mathrm{J}(s)\overline{\mathrm{J}(t)})$ for $s, t \in \G_d$.
	By Proposition \ref{prop_204}, there exist Hilbert space $\KS$, a map $L: \G_d \to \mathcal{B}(C(\DC), \mathcal{B}(\KS, \KS_2))$ and a unital $*$-representation $\rho: C(\DC) \to \mathcal{B}(\KS)$ such that 
	\begin{align*}
		L(s)(fh))^*=\rho(f)^*(L(s)(h))^* \quad \text{and} \quad \xi(s, t)(f\overline{h})=L(s)(f)(L(t)(g))^*
	\end{align*}
	for every $s, t \in \G_d$ and $f, h \in C(\DC)$. Then
	\begin{align*}	\zeta^{(12)}(s)\zeta^{(12)}(t)^*+L(s)(1)\rho(\mathrm{J}(s))\rho(\mathrm J(t))^*(L(t)(1))^*=
		\zeta^{(32)}(s)\zeta^{(32)}(t)^*+L(s)(1)(L(t)(1))^*
	\end{align*}
	for every $s, t \in \G_d$. Let us consider the map 
	\[
	V: \overline{\text{span}} \left\{
	\begin{bmatrix}
	\zeta^{(12)}(t)^*u \\
		\rho(\mathrm{J}(t))^*(L(t)(1))^*u
	\end{bmatrix}
	: u \in \KS_2,  t \in \G_d
	\right\}\to  \overline{\text{span}} \left\{
	\begin{bmatrix}
		\zeta^{(32)}(t)^*u \\
		(L(t)(1))^*u
	\end{bmatrix}
	: u \in \KS_2, t \in \G_d
	\right\}
	\]
	given by 
	\[   \begin{bmatrix}
		\zeta^{(12)}(t)^*u \\
		\rho(\mathrm{J}(t))^*(L(t)(1))^*u
	\end{bmatrix}\mapsto \begin{bmatrix}
		\zeta^{(32)}(t)^*u \\
		(L(t)(1))^*u
	\end{bmatrix}.
	\]
	Then $V$ is an isometry and so, it can be extended to a unitary from $\KS_1 \oplus \KS$ onto $\KS_3 \oplus \KS$. We can write $V=\begin{bmatrix} A & B \\ C & D \end{bmatrix}$. Let us define $\zeta^{(31)}: \G_d \to \mathcal{B}(\KS_3, \KS_1)$ as
	\[
	\zeta^{(31)}(s)^*=A+B\rho(\mathrm J(s))^*( I-D\rho(\mathrm J(s))^*)^{-1}C.
	\]
	Following the proof of Part $(3) \implies (4)$ of Theorem \ref{thm_realization_Gd}, one can show that $\zeta^{(31)} \in SA_{\G_d}(\KS_3, \KS_1)$. Also, we have 
	\[
	\zeta^{(32)}(s)^*u=A\zeta^{(12)}(s)^*u+B\rho(\mathrm J(s))^*(I-D(\rho(\mathrm J(s)))^*)^{-1}C\zeta^{(12)}(s)^*u=\zeta^{(31)}(s)^*\zeta^{(12)}(s)^*u
		\]
	for every $u \in \KS_2$ and $s \in \G_d$. The proof is now complete.
\end{proof}

Next, we consider an extension problem on $\G_d$. For a subset $U_0$ of $\G_d$, we write $\mathscr{HE}(U_0)$ for the set all bounded functions on $U_0$ having an extension to a holomorphic map in a neighbourhood of $U_0$. We consider the problem of finding necessary and sufficient conditions so that $f \in \mathscr{HE}(U_0)$ has a norm-preserving extension to $g$ in $H^\infty(\mathbb{G}_d)$ satisfying $g\slash \|f\|_{\infty, W} \in SA(\G_d)$. Such extension results are presented for the bidisc, symmetrized bidisc, tetrablock and pentablock in \cite{Agler_McCarthy_2003}, \cite{Tirtha_Sau}, \cite{Jain} and \cite{Pal2026}, respectively. Recall from Section \ref{sec_02} that $s=(s_1, \dotsc, s_{d-1}, p) \in \G_d$ if and only if  $|\Phi_\alpha(s_1, \dotsc, s_{d-1}, p)|<1$ for all $\al \in \DC$. Let $Q\G_d$ be the class of all commuting $d$-tuples $\underline{S}=(S_1, \dotsc, S_{d-1}, P)$ of Hilbert space operators with $\sigma_T(\underline{S}) \subseteq \G_d$ such that 
	\[
	\|\Phi_{\al}(S_1, \dotsc, S_{d-1}, P)\| \leq 1
	\]
		for all $\al \in \DC$. We adopt the terminology in \cite{Mittal} for quantized domains and call the set $Q\G_d$ \textit{quantum symmetrized polydisc}. As discussed in Section \ref{sec_02}, the class $\mathfrak{M}_{\G_d}$ consists of all commuting $d$-tuples $\underline{S}=(S_1, \dotsc, S_{d-1}, P)$ of operators such that 
	\[
	\|\Phi_{\al}(S_1, \dotsc, S_{d-1}, P)\| < 1
	\]
	for all $\al \in \DC$, We further showed that $\sigma_T(\underline{S}) \subseteq \G_d$ for all $\underline{S} \in \mathfrak{M}_{\G_d}$ and thus, $Q\G_d$ contains $\mathfrak{M}_{\G_d}$. For a subset $U_0$ of $\G_d$, a commuting $d$-tuple $\underline{S}=(S_1, \dotsc, S_{d-1}, P)$ is said to be \textit{subordinate} to $U_0$ if $\sigma_T(\underline{S}) \subset U_0$ and $g(\underline{S})=0$ whenever $g$ is holomorphic in a neighbourhood of $U_0$ and $g|_{U_0}=0$. Suppose $h$ is a map on $U_0$ having a holomorphic extension in a neighbourhood of $U_0$ and $\underline{S}=(S_1, \dotsc, S_{d-1}, P)$ is subordinate to $U_0$, then define $h(\underline{S})$ by setting $h(\underline{S})=g(\underline{S})$, where $g$ is any holomorphic extension of $h$ in a neighbourhood of $U_0$. The definition of $h(\underline{S})$ is independent of the choice of the holomorphic extension $g$ of $h$. We now put forth the statement of extension theorem for the domain $\G_d$. 
	
	\begin{thm}\label{thm_ext_Gd}
		Suppose $U_0$ is a subset of $\G_d$ and $f$ is a non-zero function in $\mathscr{HE}(U_0)$. Then $f$ admits a norm preserving extension $g \in H^\infty(\G_d)$ satisfying $\displaystyle \frac{1}{\|f\|_{\infty, U_0}} g \in SA(\G_d)$ if and only if 
		\[
		\|f(\underline{S})\| \leq \|f\|_{\infty, U_0} \quad \text{for every $\underline{S} \in Q\G_d$ subordinate to $U_0$}.
		\]
	\end{thm}
	
A subset $U_0$ of the symmetrized polydisc is said to have the \textit{extension property} in $SA(\G_d)$ if for every non-zero $f$ in $\mathscr{HE}(U_0)$, there exists $g \in H^\infty(\G_d)$ such that $ g\slash \|f\|_{\infty, U_0} \in SA(\G_d), g|_{U_0}=f$ and $\|f\|_{\infty, U_0}=\|g\|_{\infty, \G_d}$. As an immediate consequence of Theorem \ref{thm_ext_Gd}, we put forth the following characterization of sets in $\G_d$ having the extension property in $SA(\G_d)$. We skip its proof.

\begin{cor}
	Any set $U_0$ contained in $\G_d$ admits the extension property in $SA(\G_d)$ if and only if $\|h(\underline{S})\| \leq \|h\|_{\infty, U_0}$ for every $h \in \mathscr{HE}(U_0)$ and $\underline{S} \in Q\G_d$ subordinate to $U_0$. 
\end{cor}

The proof to the necessary part of Theorem \ref{thm_ext_Gd} is not much difficult. Let $f\in \mathscr{HE}(U_0)$ be non-zero that admits a norm preserving extension $h \in H^(\infty)(\G_d)$ such that $(1\slash \|f\|_{\infty, U_0})h \in SA(\G_d)$. Also, let $\underline{S}=(S_1, \dotsc, S_{d-1}, P) \in Q\G_d$ be subordinate to $U_0$ that acts on a Hilbert space $\mathcal{K}$. For a point $s=(s_1, \dotsc, s_{d-1}, p) \in \C^d$ and $0<r<1$, let us write $r\cdot s=(rs_1, \dotsc, r^{d-1}s_{d-1}, r^dp)$.  Let $r_n$ be an increasing sequence of non-negative numbers converging to $1$. Since $\G_d$ is $(1, \dotsc, d-1, d)$-quasi-balanced, $r_n\cdot s \in \G_d$ for all $s \in \G_d$. Let us define $h_n: \G_d \to \C$ as $h_n(s)=h(r_n\cdot s)$, which forms a uniformly bounded sequence of holomorphic maps that converge pointwise to $h$. By dominated convergence theorem, $h_n(\underline{S})$ converges to $h(\underline{S})$ under the weak operator topology. Since $\frac{1}{\|f\|_{\infty, U_0}}h \in SA(\G_d)$ and $r_n\cdot \underline{S} \in \mathfrak{M}_{\G_d}$, we have
		\[
		|\langle h(\underline{S})x_1, x_2\rangle |=\lim_{n \to \infty}|\langle h(r_n\cdot\underline{S})x_1, x_2\rangle | \leq \lim_{n \to \infty}\|h(r_n \cdot \underline{S})\| \|x_1\| \| x_2\| \leq \|f\|_{\infty, U_0} \|x\| \|y\|.
		\]
		Thus, $\|f(\underline{S})\| =\|h(\underline{S})\|\leq \|f\|_{\infty, U_0}$. For the remainder of the paper, our aim is to present a proof of the converse to Theorem \ref{thm_ext_Gd}. To do so, let us revisit Theorem \ref{thm_interpolation_Gd}, which is the interpolation theorem for $\G_d$. We denote $\textbf{s}$ by the data consisting of distinct points in $\G_d$, say, $s^{(1)}=(s_1^{(1)}, \dotsc, s_{d-1}^{(1)}, p^{(1)}), \dotsc, s^{(n)}=(s_1^{(n)}, \dotsc, s_{d-1}^{(n)}, p^{(n)})$.  We write $K_\textbf{s}$ for the collection of all $n \times n$ strictly positive definite matrices $[k(i, j)]_{i, j=1}^n$ satisfying $k(i, i)=1$ for $i=\{1, \dotsc, n\}$ such that
	\begin{equation}\label{eqn_301}
		\left[\left(1-\Phi_{\al}(s^{(i)})\overline{\Phi_{\al}(s^{(j)})}\right)k(i, j)\right] \geq 0 
		\ \text{for all} \ \al \in \DC.
	\end{equation}
Next, we re-write the $(1) \iff (2)$ part of Theorem \ref{thm_interpolation_Gd} in terms of the set $K_{\textbf{s}}$.
	
	\begin{thm}\label{thm_int_P_II}
		Let $s^{(1)},\dotsc, s^{(n)}$ be distinct points in $\G_d$, and let $\lm_1, \dotsc, \lm_n \in \DC$. Then there exists $g \in SA(\G_d)$ such that $g(s^{(i)})=\lm_i$ for $1 \leq i \leq n$ if and only if 
		\[
		\begin{bmatrix} (1-\lm_i\overline{\lm_j})k(i, j)\end{bmatrix}_{i, j=1}^n \geq 0
		\]
		for all $k \in K_\textbf{s}$.
	\end{thm} 
	
	Next, we introduce a subclass of $H^\infty(\G_d)$. For $U_0 \subseteq \G_d, f \in \mathscr{HE}(U_0)$ and $h \in H^\infty(\G_d)$, set
	\begin{align*}
	SA_f(\G_d)&=\{g \in H^\infty(\G_d): \|g(\underline{S})\| \leq \|f\|_{\infty, U_0} \ \text{for every $\underline{S} \in \mathfrak{M}_{\G_d}$} \} \ \ \text{and}\\
	\|h\|_{\mathfrak{M}_{\G_d}}&=\sup\{\|h(\underline{S})\|: \underline{S} \in \mathfrak{M}_{\G_d} \}.
	\end{align*} 
	For every $h \in SA_f(\G_d)$, it follows trivially that $\|h\|_{\infty, \G_d} \leq \|h\|_{\mathfrak{M}_{\G_d}} \leq \|f\|_{\infty, U_0}$. For the given data $\textbf{s}=\{s^{(1)}, \dotsc, s^{(n)}\} \subseteq W$ and $\boldsymbol{\lambda}=(\lambda_1, \dotsc, \lambda_n) \in \C^n$, let us define 
	\[
	\rho_f(\textbf{s}, \boldsymbol{\lambda})=\inf\left\{\|h\|_{\mathfrak{M}_{\G_d}}: h \in SA_f(\G_d) \ \text{and} \ h(s^{(i)})=\lambda_i, 1 \leq i \leq n \right\}.
	\]
The following lemma shows that the infimum $\rho_f(\textbf{s}, \boldsymbol{\lambda})$ is attained at some $h$ in $SA_f(\G_d)$. We call such a function $h$ \textit{extremal function} associated with the data $\textbf{s}$ and $\boldsymbol{\lambda}$.
	
	\begin{lem}\label{lem_304}
		Suppose $U_0 \subseteq \G_d$, and $f \in \mathscr{HE}(U_0)$ is non-zero. For given $\textbf{s}=\{s^{(1)}, \dotsc, s^{(n)}\} \subseteq U_0$ and $\boldsymbol{\lambda}=(\lambda_1, \dotsc, \lambda_n) \in \C^n$, there exists an extremal function in $SA_f(\G_d)$ associated with $\textbf{s}$ and $\boldsymbol{\lambda}$.
	\end{lem}
	
	\begin{proof}
	Suppose $\{h_m\}$ is a sequence of functions in $SA_f(\G_d)$ satisfying $h_m(s^{(i)})=\lm_i$ for $1 \leq i \leq n$ and $\rho_f(\textbf{s}, \boldsymbol{\lambda})=\underset{m \to \infty}\lim\|h_m\|_{\mathfrak{M}_{\G_d}}$. Then $\|h\_m|_{\infty, \G_d} \leq \|f\|_{\infty, U_0}$ and so, there exists a subsequence $\{h_{m_\ell}\}$ of $\{h_m\}$ converging pointwise to $h \in H^\infty(\G_d)$. Let $\underline{S} \in \mathfrak{M}_{\G_d}$. By dominated convergence theorem, $h_{m_\ell}(\underline{S})$ converges to $h(\underline{S})$ in the weak-operator topology and thus, $\|h(\underline{S})\| \leq \|f\|_{\infty, U_0}$. Also, $h(s^{(i)})=\lm_i$ for $1 \leq i \leq n$ and so, $\rho_f(\textbf{s}, \boldsymbol{\lambda}) \leq \|h\|_{\mathfrak{M}_{\G_d}}$. Let $\underline{S} \in \mathfrak{M}_{\G_d}$ be a commuting tuple of operators acting on $\mathcal{K}$. Then 
		\[
		|\langle h(\underline{S})x_1, s_2\rangle | \leq \lim_{\ell \to \infty}\|h_{m_\ell}(\underline{S})\| \|x_1\| \| x_2\| \leq \lim_{\ell \to \infty}\|h_{m_\ell}\|_{\mathfrak{M}_{\G_d}}\|x_1\|\|x_2\|=\rho_f(\textbf{s}, \boldsymbol{\lambda}) \|x_1\| \|x_2\|
		\]
		for $x_1, x_2 \in \mathcal{K}$ and thus, $\|h(\underline{S})\| \leq \rho_f(\textbf{s}, \boldsymbol{\lambda})$ for all $\underline{S} \in \mathfrak{M}_{\G_d}$. Therefore, $\|h\|_{\mathfrak{M}_{\G_d}} \leq \rho_f(\textbf{s}, \boldsymbol{\lambda})$.
	\end{proof}
	
	The following lemma provides the final step for the proof of Theorem \ref{thm_ext_Gd}.

	\begin{lem}\label{lem_305}
		Let $U_0 \subseteq \G_d$, and let $f \in \mathscr{HE}(U_0)$ be non-zero. Suppose $h \in SA_f(\G_d)$ is an extremal function associated with the data $\textbf{s}=\{s^{(1)}, \dotsc, s^{(n)}\} \subseteq U_0$ and $\boldsymbol{\lambda}=(\lambda_1, \dotsc, \lambda_n) \in \C^n$. Then there exists $\underline{S} \in Q\G_d$ subordinate to $\textbf{s}$ such that $\|h(\underline{S})\|=\rho_f(\textbf{s}, \boldsymbol{\lambda})$.
	\end{lem}

	\begin{proof}
		Let us denote by $\rho=\rho_f(\textbf{s}, \boldsymbol{\lambda})$. If $\rho=0$, then the desired conclusion follows trivially. Let us assume that $\rho>0$. Evidently, $\frac{1}{\rho}h \in SA(\G_d)$ for $1\leq i \leq n$. By Theorem \ref{thm_int_P_II},
		\[
		\left[(\rho^2-\lm_i\overline{\lm}_j)k(i, j)\right]_{i, j=1}^n \geq 0
		\] 
		for every $k \in K_\textbf{s}$. We now show that the set $K_\textbf{s}$ is compact. Clearly, $K_\textbf{s}$ is bounded as $\|k\| \leq n$ for every $k \in K_\textbf{s}$. Let $\{k_m\}$ be a sequence in $K_\textbf{s}$ that converges to $k$ in the matrix norm. Evidently, $k$ is positive semi-definite and the conditions as in \eqref{eqn_301} hold for $k$. It is only left to prove that $k$ is invertible. Assume on the contrary that there is a non-zero vector $w=(w_1, \dotsc, w_n)^t \in \C^n$ such that $kw=0$. Let $M_1, \dotsc, M_{d-1}, M_d$ be  $n \times n$ diagonal matrices whose $(i, i)$-th entries are $s_1^{(i)}, \dotsc, s_{d-1}^{(i)}, p^{(i)}$, respectively. By \eqref{eqn_301}, $k(M_1w)=\dotsc=k(M_{d-1}w)=k(M_dw)=0$ and so, $k(h(M_1, \dotsc, M_{d-1}, M_d)w)=0$ for any holomorphic polynomial $g$ in $d$ many variables. One can choose $1 \leq i \leq n$ so that $w_i \ne 0$, and a polynomial $g$ such that $g(s^{(i)})=1$ and $g(s^{(j)})=0$ for $1 \leq j \leq n$ with $j$'s distinct from $i$. Then $k(g(M_1, \dotsc, M_{d-1}, M_d)w)$ is a scalar multiple of the $(i, i)$-th column of $k$. Hence, $k(i, i)=0$, a contradiction to the fact that $k(i, i)=1$. Hence, $K_\textbf{s}$ is compact. Next, consider the collection 
		\[
		\Upsilon=\left\{\lambda : 0 < \lambda \leq \|f\|_{\infty, U_0} \ \ \text{and} \ \ \left[(\lambda^2-\lm_i\overline{\lm}_j)k(i, j)\right]_{i, j=1}^n \geq 0 \ \text{for all} \ k \in K_\textbf{s} \right\}
		\]
		which is non-empty as $\rho \in \Upsilon$. Let $\lambda \in \Upsilon$. It follows from Theorem \ref{thm_int_P_II} that there exists $g \in SA(\G_d)$ such that $g(s^{(i)})=\lm_i\slash \lm$ for $1 \leq i \leq n$. If we define $g_*=\lm g$, then $\|g_*(\underline{S})\| \leq \lm \leq \|f\|_{\infty, U_0}$ for every $\underline{S} \in \mathfrak{M}_{\G_d}$ and so, $g_* \in SA_f(\G_d)$. Clearly,  $g_*(s^{(i)})=\lm_i$ for each $i$. Then $\rho \leq \|g_*\|_{\mathfrak{M}_{\G_d}} \leq \lambda$. Consequently, $\rho \leq \lambda$ for every $\lambda \in \Upsilon$ and hence $\rho=\inf \Upsilon$. Our next step is to prove the existence of a matrix $A \in K_\textbf{s}$ and a non-zero vector $y=(y_1, \dotsc, y_n)^t$ such that
		\begin{equation}\label{eqn_302}
			\overset{n}{\underset{i, j=1}{\sum}}(\rho^2-\lm_i\overline{\lm}_j)A(i, j)\overline{y}_iy_j=0.
		\end{equation}
	For a given $k \in K_{\textbf{s}}$, let $m(k)$ be the minimum eigenvalue of the matrix $[(\rho^2-\lm_i\overline{\lm}_j)k(i, j)]_{i, j=1}^n$ and let $\mu_*=\inf\{m(k): k \in K_\textbf{s}\}$. Suppose $\mu_*>0$. Set $M=\sup\{\|k\|: k \in K_\textbf{s}\}$, which is finite and $0<\epsilon< \mu_* \slash M$. Then for every $y \in \C^n$ and $k \in K_\textbf{s}$,
		\[
		\left\langle \left[(\rho^2-\lm_i\overline{\lm}_j)k(i, j)\right]_{i, j=1}^n y, y \right\rangle \geq m(k)\|y\|^2 \geq \mu_* \|y\|^2
		\ \ \text{and so,} \ \
		\left[(\rho^2-\epsilon-\lm_i\overline{\lm}_j)k(i, j)\right]_{i, j=1}^n \geq 0
		\]
		for every $k \in K_\textbf{s}$, contradicting the minimality of $\rho$. Hence, $\mu_*=0$ and the claim as in \eqref{eqn_302} follows from the compactness of $K_{\textbf{s}}$ and the continuity of the map $k \mapsto m(k)$. Let $A=[A(i, j)] \in K_\textbf{s}$ and $y$ be non-zero vector in $\C^n$ such that \eqref{eqn_302} holds. Since $A$ is invertible, the column space $\mathcal{K}_n=\text{span}\{A(., j) : 1\leq j \leq n \}$ is an $n$-dimensional space. Define $(S_1, \dotsc, S_{d-1}, P)$ on $\mathcal{K}_n$ as
		\[
		S_1^*A(., j)=\overline{s}_1^{(j)}A(., j), \quad \dotsc, \quad S_{d-1}^*A(., j)=\overline{s}_{d-1}^{(j)}A(., j) \quad \text{and} \quad P^*A(., j)=\overline{p}^{(j)}A(., j)
		\]
		for $1 \leq j \leq n$. Evidently, $(S_1^*, \dotsc, S_{d-1}^*, P^*)$ is a commuting tuple of operators with the joint spectrum $\sigma_T(S_1^*, \dotsc, S_{d-1}^*, P^*)=\{(\overline{s}_1^{(j)}, \dotsc, \overline{s}_{d-1}^{(j)}, \overline{p}^{(j)}) : 1\leq j \leq n\}$. Thus, $\underline{S}=(S_1, \dotsc, S_{d-1}, P)$ is subordinate to $\textbf{s}$. In fact, $\underline{S}$ and $\underline{S}^*=(S_1^*,\dotsc,  S_{d-1}^*, P^*)$ belong to $Q\G_d$. Also, $g(\underline{S})^*A(., j)=\overline{g(s^{(j)})}A(., j)=\overline{\lm}_jA(., j)$ for $1 \leq j \leq n$. Since $\left[(\rho^2-\lm_i\overline{\lm}_j)A(i, j)\right] \geq 0$, we have that $\|h(\underline{S})\|=\|h(\underline{S})^*\| \leq \rho$. The desired conclusion now follows directly from \eqref{eqn_302}.
	\end{proof}
	
	Finally, we present a proof to the extension theorem for $\G_d$.
	
	\medskip
	
	\noindent \textit{Proof of Theorem \ref{thm_ext_Gd}:} Let $U_0 \subseteq \G_d$, and let $f \in \mathscr{HE}(U_0)\setminus \{0\}$ such that 
	\begin{equation}\label{eqn_303}
		\|f(\underline{S})\| \leq \|f\|_{\infty, U_0} \quad \text{for every $\underline{S} \in Q\G_d$ subordinate to $U_0$}.	
	\end{equation}
	Let  $\{s^{(1)}, s^{(2)}, \dotsc\}$ be a dense countable subset of $U_0$. Denote by $\textbf{s}_n=\{s^{(1)}, \dotsc, s^{(n)}\}$ and  $\boldsymbol{\lambda}_n=(f(s^{(1)}), \dotsc, f(s^{(n)}))$. By Lemma \ref{lem_304}, there exists an extremal function $h_n \in SA_f(\G_d)$ that satisfies 
	\[
	\rho_f(\textbf{s}_n, \boldsymbol{\lambda}_n)=\|h_n\|_{\mathfrak{M}_{\G_d}}
	\] for each $n$. It follows from Lemma \ref{lem_305} that there is an operator tuple $\underline{S}_n=(S_1^{(n)}, \dotsc, S_{d-1}^{(n)}, P^{(n)})$ in $Q\G_d$ subordinate to $\textbf{s}_n$ such that $\|h_n(\underline{S}_n)\|=\rho_f(\textbf{s}_n, \boldsymbol{\lambda}_n)$. Therefore, 
	\begin{align*}
		\|h_n\|_{\mathfrak{M}_{\G_d}}
		=\rho_f(\textbf{s}_n, \boldsymbol{\lambda}_n)
		=\|h_n(\underline{S}_n)\|  
		=\|f(\underline{S}_n)\|
		\leq \|f\|_{\infty, U_0},
	\end{align*}
	where the last inequality follows since $\underline{S}_n$ is subordinate to $W$, $\textbf{s} \subseteq W$ and $f$ satisfies \eqref{eqn_303}.
	Since $\|h_n\|_{\infty, \G_d} \leq \|f\|_{\infty, U_0}$, the sequence $\{h_n\}$ is uniformly bounded. Therefore, there exists a subsequence $\{h_{n_\ell}\}$ of $\{h_n\}$ converging pointwise to a function $h \in H^\infty(\G_d)$. Evidently,  $f(s^{(i)})=h(s^{(i)})$ for $i \in \N$ and so, $f=g$ on $U_0$. Also, $\|f\|_{\infty, U_0}=\|h\|_{\infty, U_0} \leq \|h\|_{\infty, \G_d}$. For every $\underline{S}$ in $\mathfrak{M}_{\G_d}$, an application of dominated convergence theorem gives that $h_{n_\ell}(\underline{S})$ converges to $h(\underline{S})$ in the weak-operator topology. Thus, $\|h(\underline{S})\| \leq \|f\|_{\infty, U_0}$ and so, $\|h\|_{\infty, \G_d} \leq \|f\|_{\infty, U_0}$. \qed

\section{Function Theory on the Generalized Symmetrized Domain $\Theta_d$}\label{sec_04}	

\noindent The symmetrization map $\pi_d=(s_1, \dotsc, s_d): \C^d \to \C^d$ is defined by
\[
s_i(z_1, \dotsc, z_d)=\underset{1 \leq \ell_1<\dotsc <\ell_i  \leq d}{\sum}z_{\ell_1}\dotsc z_{\ell_i} \quad (1 \leq i \leq d).
\]	
As mentioned earlier, the symmetrized polydisc $\G_d$ is the image of the polydisc $\D^d$ under the symmetrization map. The authors of \cite{Biswas} introduced a family of domains that generalize the symmetrized polydisc. Fix $m \in \mathbb{N}$ and let $p$ be a divisor of $m$. The domain $\Theta_d$ is defined as
\[
\Theta_d=\left\{(\theta_1, \dotsc, \theta_d) \in \C^d: \theta_i=s_i(z_1^m, \dotsc, z_d^m), \dotsc,  \theta_d=(z_1 \dotsc z_d))^{m\slash p}, 1 \leq i \leq d,  z_1, \dotsc, z_d \in \D \right\}.
\]	
If $m=p=1$, then the corresponding domain $\Theta_d$ is precisely the symmetrized polydisc $\G_d$. Although the domain $\Theta_d$ depends on the parameters $m$ and $p$, we suppress this dependence in the notation for convenience.

\smallskip 

In this section, we present realization, interpolation, Toeplitz corona and 	extension theorems for the domain $\Theta_d$. Since the proofs follow essentially the same arguments as those for $\mathbb{G}_d$, we only provide the necessary constructions and state the corresponding results. To begin with, we recall from \cite{Keshari} some useful characterizations of the domain $\Theta_d$. For a point $\theta=(\theta_1, \dotsc, \theta_d) \in \C^d$, consider the polynomials $P_*(\theta), Q_*(\theta)$ and $R_*(\theta)$ given by 
\begin{align*}
P_*(\theta)(\alpha)
&=\alpha^d-\theta_1\alpha^{d-1}+\dotsc+(-1)^d\theta_d^p,
\\
Q_*(\theta)(\alpha)
&=\frac{d}{d\alpha}(\alpha^dP_*(\theta)(1\slash \alpha))
=d(-1)^d\theta_d^p\alpha^{d-1}+(d-1)(-1)^{d-1}\theta_{d-1}\alpha^{d-2}+\dotsc+(-\theta_1),  
\\
R_*(\theta)(\alpha)
&=\alpha^{d-1}P_*'(\theta)(1\slash \alpha)
=d-(d-1)\theta_1\alpha+\dotsc+(-1)^{d-1}\theta_{d-1}\alpha^{d-1},
\end{align*}
which are analogs of the polynomials defined in \eqref{eqn_QR}. Using $Q_*(\theta)$ and $R_*(\theta)$, we define 
\[
\mathrm J_*(\theta)(\alpha)=Q_*(\theta)(\al)\slash R_*(\theta)(\al).
\] 
For every $\alpha \in \DC$ and $\theta=(\theta_1, \dotsc, \theta_d) \in \C^d$, define a rational function as
\[
\Phi_{*, \alpha}(\theta_1, \dotsc, \theta_d)=\mathrm J_*(\theta)(\alpha)=\frac{d(-1)^d\theta_d^p\alpha^{d-1}+(d-1)(-1)^{d-1}\theta_{d-1}\alpha^{d-2}+\dotsc+(-\theta_1)}{d-(d-1)\theta_1\alpha+\dotsc+(-1)^{d-1}\theta_{d-1}\alpha^{d-1}}.
\]	
By Theorem 3.1 in \cite{Keshari}, $(\theta_1, \dotsc, \theta_d) \in \Theta_d$ if and only if  $\sup\{|\Phi_{*, \al}(\theta_1, \dotsc, \theta_d)|: \alpha \in \DC\}<1$. Following the proof of Theorem 3.1 in \cite{Keshari} and the arguments similar to that to Theorem \ref{prop_char}, one can easily prove that 
\[
(\theta_1, \dotsc, \theta_d) \in \Theta_d \ \text{if and only if}  \ |\Phi_{*, \al}(\theta_1, \dotsc, \theta_d)|<1 \ \text{for every $\alpha \in \DC$.}
\]
The functions $\Phi_{*,\alpha}$ and $\mathrm J_*(\theta)$ associated with $\Theta_d$ serve as analogs of the functions $\Phi_{\alpha}$ and $\mathrm J(s)$, respectively, for the domain $\G_d$.  Next, we define a class of commuting $d$-tuples of operators as 
\[
\mathfrak{M}_{\Theta_d}=\left\{\underline{T}=(T_1, \dotsc, T_d): \sigma_T(\underline{T}) \subseteq \Theta_d, \ \|\Phi_{*, \al}(\underline{T})\|<1 \ \text{for every} \ \al \in \DC\right\}.
\]
The class $\mathfrak{M}_{\Theta_d}$ plays the same role for $\Theta_d$ as $\mathfrak{M}_{\G_d}$ does for $\G_d$. In the same spirit for $\G_d$, we now define the Schur-Agler class and admissible kernels for the domain $\Theta_d$. 

\begin{defn}
	For Hilbert spaces $\LS, \LS'$, the \textit{$\mathcal{B}(\LS, \LS')$-valued Schur-Agler class} is defined as
	\[
	SA_{\Theta_d}(\LS, \LS')=\left\{f: \Theta_d \to \mathcal{B}(\LS, \LS'): \ \text{$f$ is holomorphic}, \ \|f(\underline{T})\| \leq 1 \ \text{for every $\underline{T} \in \mathfrak{M}_{\Theta_d}$}\right\}.
	\] 
We simply denote the class $SA(\Theta_d)(\C, \C)$ by $SA(\Theta_d)$. A $\mathcal{B}(\LS)$-valued kernel (or, weak kernel) $k$ on a subset of $\Theta_d$ is said to be \textit{admissible} if
\[
(\theta, \widetilde{\theta}) \mapsto \left(1-\Phi_{*, \al}(\theta)\overline{\Phi_{*, \al}(\widetilde \theta)}\right)k(\theta, \widetilde{\theta})
\]
is a positive semi-definite map for all $\alpha \in \DC$. For a subset $F$ of the domain $\Theta_d$, a map $\xi: F \times F \to \mathcal{B}(C(\DC), \mathcal{B}(\LS))$ is called \textit{completely positive} if for every $n \in \N, \{v_1, \dotsc, v_n\} \subset \LS, \{s^{(1)}, \dotsc, s^{(n)}\} \subset F$ and $\{h_1, \dotsc, h_n\} \subset C(\DC)$, we have
\[
\overset{n}{\underset{i, j=1}{\sum}}\la \xi(s^{(i)}, s^{(j)})(h_i\overline{h}_j)v_j, v_i\ra_\LS \geq 0.
\]	
\end{defn}	

Having introduced the necessary terminology and definitions, we now present the realization theorem for $\Theta_d$. The proof follows the same argument as that of Theorem \ref{thm_realization_Gd}, with minor modifications. We leave the details to the reader.

	\begin{thm}\label{thm_realization_Theta_d}
	Let $\LS_1$ and  $\LS_2$ be Hilbert spaces. For a function $f: \Theta_d \to \mathcal{B}(\LS_1, \LS_2)$, the following statements are equivalent: 
	\begin{enumerate}[leftmargin=*]
		\item[$(1)$] $f \in SA_{\Theta_d}(\LS_1, \LS_2)$;
		
		\item[$(2)$] $(\theta, \widetilde \theta) \mapsto (I_{\LS_2}-f(\theta)f(\widetilde \theta)^*)\oslash k(\theta, \widetilde \theta)$ is a positive semi-definite function for all $\mathcal{B}(\LS_2)$-valued admissible kernel $k$ on $\Theta_d$;
		
		\item[$(3)$] there exists a completely positive map  $\xi: \Theta_d \times \Theta_d \to \mathcal{B}(C(\DC), \mathcal{B}(\LS_2))$ such that
		\[
		I_{\LS_2}-f(\theta)f(\widetilde \theta)^*=\xi(\theta, \widetilde \theta)(1-\mathrm{J}_*(\theta)\overline{\mathrm{J}_*(\widetilde \theta)})
		\]
		for all $\theta, \widetilde \theta \in \Theta_d$;
		
		\item[$(4)$] there exist a Hilbert space $\HS$, a unital $*$-representation $\rho: C(\DC) \to \mathcal{B}(\HS)$ and a unitary $V=\begin{bmatrix} A & B \\ C & D \end{bmatrix}: \LS_1 \oplus \HS \to \LS_2 \oplus \HS$ such that for every $\theta \in \Theta_d$,
		\[
		f(\theta)=A+B\rho(\mathrm{J}_*(\theta))(I_\HS-D\rho(\mathrm{J}_*(\theta)))^{-1}C. 
		\]
	\end{enumerate}
\end{thm}	

We next turn to the interpolation theorem for $\Theta_d$. We skip its proof as it follows from the corresponding result for $\G_d$ (see Theorem \ref{thm_interpolation_Gd}) with necessary modifications. 

\begin{thm}\label{thm_interpolation_Theta_d}
	Assume that $\LS_1, \LS_2$ are Hilbert spaces. Let $F=\{\theta^{(1)}, \dotsc, \theta^{(n)}\} \subseteq \Theta_d$, and let $A_1, \dotsc, A_n \in \mathcal{B}(\LS_1, \LS_2)$. Then the following are equivalent:
	\begin{enumerate}[leftmargin=*]
		\item[$(1)$] there exists a function $f \in SA_{\Theta_d}(\LS_1, \LS_2)$ such that $f(\theta^{(i)})=A_i$ for $1 \leq i \leq n$;
		\item[$(2)$] $\begin{bmatrix} (I_{\LS_2}-A_iA_j^*)\otimes k(\theta^{(i)}, \theta^{(j)})\end{bmatrix}_{i, j=1}^n \geq 0$ for every $\mathcal{B}(\LS_2)$-valued kernel $k$ on $\Theta_d$;
		\item[$(3)$] there exists a completely positive map  $\xi: \Theta_d \times \Theta_d \to \mathcal{B}(C(\DC), \mathcal{B}(\LS_2))$ such that
		\[
		I_{\LS_2}-A_iA_j^*=\xi(\theta^{(i)}, \theta^{(j)})(1-\mathrm{J}_*(\theta^{(i)})\overline{\mathrm{J}_*(\theta^{(j)})})
		\]
		for $1 \leq i, j \leq n$.
	\end{enumerate}
\end{thm}	

We proceed to the Toeplitz corona theorem for $\Theta_d$. As in the case of $\G_d$, the result follows by applying the same method with appropriate modifications. We omit the proof and refer the reader to the corresponding arguments from Theorem \ref{thm_TC_Gd}.	
	
\begin{thm}\label{thm_TC_Theta_d}
	Let $\LS_1, \LS_2$ and $\LS_3$ be Hilbert spaces. For functions $\zeta^{(12)}: \Theta_d \to \mathcal{B}(\LS_1, \LS_2)$ and $\zeta^{(32)}: \Theta_d \to \mathcal{B}(\LS_3, \LS_2)$, the following are equivalent: 
	\begin{enumerate}[leftmargin=*]
		\item[$(1)$] there exists $\zeta^{(31)} \in SA_{\Theta_d}(\LS_3, \LS_1)$ such that $\zeta^{(12)}(\theta)\zeta^{(31)}(\theta)=\zeta^{(32)}(\theta)$ for every $\theta \in \Theta_d$;
		\item[$(2)$] for every finite subset $F=\{\theta^{(1)}, \dotsc, \theta^{(n)}\} \subset \Theta_d$,
		\[
		\begin{bmatrix}
			\left(\zeta^{(12)}(\theta^{(i)})\zeta^{(12)}(\theta^{(j)})^*-\zeta^{(32)}(\theta^{(i)})\zeta^{(32)}(\theta^{(j)})^*\right)\otimes k(\theta^{(i)}, \theta^{(j)})
		\end{bmatrix}_{i, j=1}^n \geq 0
		\]
		for every $\mathcal{B}(\LS_2)$-valued admissible kernel $k$ on $\Theta_d$;
		\item[$(3)$] there exist a completely positive map $\xi: \Theta_d \times \Theta_d \to \mathcal{B}(C(\DC), \mathcal{B}(\LS_2))$ such that 
		\begin{align*}
			\zeta^{(12)}(\theta)\zeta^{(12)}(\widetilde \theta)^*-\zeta^{(32)}(\theta)\zeta^{(32)}(\widetilde \theta)^*
			=\xi(\theta, \widetilde{\theta})(1-\mathrm{J}_*(\theta)\overline{\mathrm{J}_*(\widetilde \theta)})
		\end{align*}
		for all $\theta, \widetilde \theta \in \Theta_d$.	 	
	\end{enumerate}
\end{thm}		
	
We now study the following extension problem on $\Theta_d$: given a subset $W \subseteq \Theta_d$, characterize functions $f \in \mathscr{HE}(W)$ that admits a norm-preserving extension to a function $g$ in $H^\infty(\Theta_d)$ such that $g\slash \|f\|_{\infty, W} \in SA(\Theta_d)$. Let $Q\Theta_d$ be the class of all commuting $d$-tuples $\underline{T}=(T_1, \dotsc, T_d)$ of operators with $\sigma_T(\underline{T}) \subseteq \Theta_d$ such that
\[
\|\Phi_{*, \al}(T_1, \dotsc, T_d)\| \leq 1
\]
for all $\al \in \DC$. For a subset $W$ of $\Theta_d$, we say that a commuting $d$-tuple $\underline{T}$ is \textit{subordinate} to $W$ if $\sigma_T(\underline{S}) \subset W$ and $g(\underline{S})=0$ whenever $g$ is holomorphic in a neighbourhood of $W$ and $g|_W=0$. If $f$ is a function on $W$ that admits a holomorphic extension in a neighbourhood of $W$ and $\underline{T}$ is subordinate to $W$, then define $f(\underline{T})$ by setting $f(\underline{T})=g(\underline{T})$, where $g$ is any holomorphic extension of $f$ in a neighbourhood of $W$. With these definitions in place, we present the extension theorem for $\Theta_d$. The proof follows along the same lines as the corresponding result for $\G_d$ (see Theorem \ref{thm_ext_Gd})  with only minor modifications.

\begin{thm}\label{thm_ext_Theta_d}
	Let $W \subseteq \Theta_d$, and let $f \in \mathscr{HE}(W)$ be non-zero. Then there exists $g \in H^\infty(\Theta_d)$ such that $\displaystyle (1\slash \|f\|_{\infty, W}) g \in SA(\Theta_d), g|_W=f$ and $\|g\|_{\infty, \Theta_d}=\|f\|_{\infty, W}$ if and only if $\|f(\underline{T})\| \leq \|f\|_{\infty, W}$ for every $\underline{T} \in Q\Theta_d$ subordinate to $W$.
\end{thm}

	\noindent \textbf{Funding.} The first named author is supported in part by the “Core Research Grant” with Award No. CRG/2023/005223 from Anusandhan National Research Foundation (ANRF) of Govt. of India. The second author is supported by a postdoctoral fellowship (Ref. No. 0204/9(10)/2026-R\&D-II/4805) from the National Board for Higher Mathematics (NBHM), Government of India.

\end{document}